\documentclass[11pt,reqno]{amsart}

 \usepackage{amsfonts,amsmath,amscd,amssymb,amsbsy,amsthm,amstext,amsopn,amsxtra}
 \usepackage{color,fullpage,mathrsfs}
 \usepackage{subfigure}
 \usepackage{verbatim} 
 \usepackage{stmaryrd}
 \usepackage{extarrows}
 \usepackage[all]{xy}
 \usepackage{bm}   
 \usepackage{hyperref}
 \usepackage{tikz}
 \usepackage{tikz-cd}
 \usepackage{enumitem}
 \usetikzlibrary{calc}

\usepackage[normalem]{ulem} 

\usepackage{cancel}
\usepackage{color}
\usepackage{tikz}

 \numberwithin{equation}{section}
 \hypersetup{colorlinks,linkcolor=[rgb]{0,0,1},citecolor=[rgb]{1,0,0}}

\newtheorem{theorem}{Theorem}[section]
\newtheorem{lemma}[theorem]{Lemma}
\newtheorem{proposition}[theorem]{Proposition}
\newtheorem{corollary}[theorem]{Corollary}
\newtheorem{conjecture}[theorem]{Conjecture}
\newtheorem{assumption}[theorem]{Assumption}

\theoremstyle{definition}

\newtheorem{example}[theorem]{Example}

\newtheorem{remark}[theorem]{Remark}

\newcommand{\one}{\ensuremath{(\mathrm{i})}}
\newcommand{\two}{\ensuremath{(\mathrm{ii})}}

\newcommand{\CC}{\ensuremath{\mathbb{C}}} 

\newcommand{\QQ}{\ensuremath{\mathbb{Q}}}

\newcommand{\ZZ}{\ensuremath{\mathbb{Z}}}

\newcommand{\chamber}{C}

\newcommand{\cont}{\operatorname{cont}}

\newcommand{\Dnil}{D_{\operatorname{nil}}}

\newcommand{\End}{\operatorname{End}}
\newcommand{\Ext}{\operatorname{Ext}}
\newcommand{\ghilb}{\ensuremath{G}\operatorname{-Hilb}}

\newcommand{\GL}{\operatorname{GL}}

\newcommand{\Hom}{\operatorname{Hom}}
\newcommand{\0}{\operatorname{0}}
\newcommand{\I}{\operatorname{I}}
\newcommand{\II}{\operatorname{II}}
\newcommand{\III}{\operatorname{III}}
\newcommand{\id}{\operatorname{id}} 
\newcommand{\Irr}{\operatorname{Irr}}

\newcommand{\Knum}{K^{\operatorname{num}}}

\newcommand{\Nef}{\operatorname{Nef}}

\newcommand{\Pic}{\operatorname{Pic}}

\newcommand{\rank}{\operatorname{rk}}
\newcommand{\relint}{\operatorname{relint}}
 \DeclareMathOperator{\Rderived}{\mathbf{R}\!}

\newcommand{\Seshadri}{\operatorname{S}}
\newcommand{\SL}{\operatorname{SL}}

\newcommand{\supp}{\operatorname{supp}}

\title{The Cautis--Logvinenko conjecture}

\author{Alastair Craw} 
\address{Department of Mathematical Sciences, 
University of Bath, 
Claverton Down, 
Bath BA2 7AY, 
UK.}
\email{a.craw@bath.ac.uk}
\urladdr{https://sites.google.com/view/alastair-craw/}

\author{Hsueh-Yung Lin}
\address{National Taiwan University,
Department of Mathematics,
No. 1, Sec. 4, Roosevelt Rd.,
Taipei 10617, Taiwan}
\email{hsuehyunglin@ntu.edu.tw}

\author{Ryo Yamagishi} 
\address{Department of Mathematics, Faculty of Science, Kyushu University, 744, Motooka, Nishi-ku, Fukuoka, Japan 819-0395.}
\email{yamagishi@kyushu-u.ac.jp}

\begin{document}
\begin{abstract}
 For a finite subgroup $G\subset \SL(3,\CC)$, the Cautis--Logvinenko conjecture states that for each nontrivial irreducible representation $\rho$ of $G$, the image of the sheaf $\mathcal{O}_0\otimes \rho$ under the derived equivalence of Bridgeland--King--Reid is a pure sheaf on the $G$-Hilbert scheme. We prove a strong form of this conjecture in complete generality, and in doing so, we compute the relevant sheaf explicitly whenever
 its support is of dimension one. Our main result implies that a matrix defining the Gale dual of the linearisation map is sign-coherent, thereby allowing us to read off the support and cohomological degree of the pure sheaves directly from the matrix. 
\end{abstract}

 \maketitle

\section{Introduction}
   For a finite subgroup $G\subset \SL(2,\mathbb{C})$, the McKay correspondence \cite{McKay80} is a bijection between the nontrivial irreducible representations of $G$ and the irreducible components of the exceptional locus of the minimal resolution $\tau\colon S\to \mathbb{C}^2/G$. For a geometric construction of this correspondence, we regard the surface $S$ as the $G$-Hilbert scheme of $\mathbb{C}^2$ following Ito--Nakamura~\cite{ItoNakamura99}, in which case the resulting universal family over $S$ determines an equivalence from the derived category of coherent sheaves on $S$ to the derived category of $G$-equivariant coherent sheaves on $\mathbb{C}^2$ by \cite{KV00} (see also \cite{BKR01}). For every nontrivial irreducible representation $\rho$ of $G$, the inverse derived equivalence
 \[
 \Psi\colon D^G(\mathbb{C}^2)\stackrel{\sim}{\longrightarrow} D(S)
 \]
 sends the $G$-equivariant coherent sheaf $\mathcal{O}_0\otimes \rho$ supported at the origin in $\mathbb{C}^2$ to the twist by $-1$ of the structure sheaf of an irreducible component of the exceptional locus of $\tau$, and the assignment
 \begin{equation}
     \label{eqn:geometricMcKay}
\rho \longmapsto \supp\Psi(\mathcal{O}_{0}\otimes \rho)
  \end{equation}
 provides a natural geometric realisation of the classical McKay bijection.
   
    The search for an analogue of the McKay correspondence for a finite subgroup $G\subset \SL(3,\mathbb{C})$ in the spirit of bijection \eqref{eqn:geometricMcKay} was initiated by Cautis--Logvinenko~\cite{CL09}. The work of Bridgeland, King and Reid~\cite{BKR01} establishes that the $G$-Hilbert scheme is a crepant resolution $\tau\colon Y:= \ghilb(\mathbb{C}^3)\to \mathbb{C}^3/G$, and moreover, the universal family over $Y$ determines an equivalence from the derived category of coherent sheaves on $Y$ to the derived category of $G$-equivariant coherent sheaves on $\mathbb{C}^3$; write
 \[
 \Psi\colon D^G(\mathbb{C}^3)\stackrel{\sim}{\longrightarrow} D(Y)
 \]
 for the inverse to the BKR derived equivalence. Every irreducible representation $\rho$ of $G$ determines a complex of coherent sheaves $\Psi(\mathcal{O}_0\otimes \rho)$ whose support is contained in the exceptional locus of $\tau$ (see Corollary~\ref{cor:PsiSrhofibre}).  

 In this context, Cautis--Logvinenko~\cite[Conjecture~1.2]{CL09} made the following prediction:

 \begin{conjecture}[Cautis--Logvinenko]
 \label{conj:CLintro}
  Let $G\subset \SL(3,\mathbb{C})$ be a finite subgroup. For each nontrivial, irreducible representation $\rho$ of $G$, the object $\Psi(\mathcal{O}_0\otimes \rho)$ is a pure sheaf on $Y=\ghilb(\mathbb{C}^3)$. More precisely, there is a unique $k\in \{-1,0\}$ depending on $\rho$ such that $H^{k}\big(\Psi(\mathcal{O}_0\otimes \rho)\big)\neq 0$.
 \end{conjecture}
    
   As evidence, Cautis--Logvinenko~\cite[Theorem~1.1]{CL09} proved the conjecture for all cyclic subgroups $G$ for which the quotient  $\mathbb{C}^3/G$ has an isolated singularity \cite{CL14}, and this was later extended to all finite abelian subgroups of $G\subset \SL(3,\mathbb{C})$ by Cautis, Craw and Logvinenko~\cite{CCL17}, where toric geometry allowed for an explicit description of the pure sheaves $\Psi(\mathcal{O}_0\otimes \rho)$ for every irreducible representation $\rho$ of $G$. However, evidence for the conjecture in the nonabelian case is sparse at best: work of Van den Bergh~\cite{VdB04Duke}, which predates \cite{CL09}, implies that the conjecture holds in the very special case where the fibres of the crepant resolution $\tau$ have dimension at most one (see \cite[Example~4.6]{C21}).

   \medskip
   
 Our main result establishes a strong form of Conjecture~\ref{conj:CLintro} 
  in complete generality. Even in the abelian case, our approach simplifies significantly the proofs of the conjecture from \cite{CCL17, CL09} because we bypass the technical analysis of a resolution from \cite{CQV15} on which the original proofs rely.
 
 To state our main result, let $R(G)$ denote the representation ring of $G$, write $v$ (resp.\ $\rho_0$) for the regular (resp.\ trivial) representation of $G$, and consider the $\QQ$-vector space 
 \[
 \Theta:=\{\theta\in \Hom_{\mathbb{Z}}(R(G),\mathbb{Q}) \mid \theta(v)=0\}.
 \]
 The skew group algebra $A:= \mathbb{C}[x,y,z]\rtimes G$ is Morita equivalent to a quotient of the path algebra of the McKay quiver $Q$ of $G$ by a two-sided ideal of relations. In light of Ito--Nakajima~\cite{ItoNakajima00} and an observation of Alastair King, the $G$-Hilbert scheme is a fine moduli space $\ghilb(\mathbb{C}^3)\cong \mathcal{M}_\theta(A,v)$ 
  of $\theta$-stable $A$-modules of dimension vector $v$, where $\theta\in \Theta$ is any stability condition in the unique GIT chamber $C$ in $\Theta$ that contains the positive orthant $\big\{\theta\in \Theta \mid \theta_\rho:=\theta(\rho) > 0 \text{ for }\rho\neq \rho_0\big\}$. The coordinate hyperplanes $(\theta_\rho=0)$ in $\Theta$ indexed by the nontrivial irreducible representations $\rho$ of $G$ provide all supporting hyperplanes for the positive orthant. However, $(\theta_\rho=0)$ is not a supporting hyperplane of $C$ if and only if $(\theta_\rho=0)$ intersects $C$ (see Lemma~\ref{lem:C}, while \cite[Example~9.6]{CI04} shows a nontrivial example).
  Our main result can be stated as follows:
  
 \begin{theorem}
 \label{thm:mainintro}
    Let $G\subset \mathrm{SL}(3,\mathbb{C})$ be a finite subgroup.  For each nontrivial, irreducible representation $\rho$ of $G$, precisely one of the following occurs:
    \begin{enumerate}
      \item[\ensuremath{(+)}] the hyperplane $(\theta_\rho=0)$ in $\Theta$ does not intersect $C$, and       $\Psi(\mathcal{O}_0\otimes \rho)$ is a coherent sheaf whose support is a connected, proper surface in $Y;$ 
      \item[\ensuremath{(0)}] the hyperplane $(\theta_\rho=0)$ in $\Theta$ does not intersect $C$, and 
      there is a unique nonsingular, rational curve $\ell\cong \mathbb{P}^1$ in $Y$ such that $\Psi(\mathcal{O}_0\otimes \rho)\cong \mathcal{O}_{\ell}(-1);$ or 
     \item[\ensuremath{(-)}]  the hyperplane $(\theta_\rho=0)$ in $\Theta$ intersects $C$, and the complex $\Psi(\mathcal{O}_0\otimes \rho)[-1]$ is a coherent sheaf whose support is a connected proper surface in $Y$.    
     \end{enumerate}  
  Thus, Conjecture~\ref{conj:CLintro} holds, where $k=-1$ if and only if  the hyperplane $(\theta_\rho=0)$ intersects $C$.
  \end{theorem}

 It is important to observe that the object $\Psi(\mathcal{O}_0\otimes \rho_0)$ associated to the trivial representation can have more than one nonvanishing cohomology sheaf. The simplest example occurs for the cyclic group action of type $\frac{1}{6}(1,2,3)$, where the exceptional locus of $\tau$ has two irreducible components comprising one curve and one surface. In this case, Cautis,  Craw and Logvinenko~\cite[Theorem~1.2]{CCL17} implies that  $H^k(\Psi(\mathcal{O}_0\otimes \rho_0))\neq 0$ for $k=-2$ and $k=-1$, so the assumption that $\rho$ is nontrivial in Theorem~\ref{thm:mainintro} is essential.

 \medskip

The central role played by Conjecture~\ref{conj:CLintro} in extending the classical McKay correspondence in the spirit of \eqref{eqn:geometricMcKay} to finite nonabelian subgroups of $\SL(3,\mathbb{C})$ was highlighted in recent work by the first author~\cite[Section~6]{C21}, and Theorem~\ref{thm:mainintro} allows us to build on that work. For this, consider the dimension filtration $\{0\}\subset F_0\subseteq F_1\subseteq F_2=K_0(Y)_{\QQ}$ of the Grothendieck group of coherent sheaves with support in $\tau^{-1}(\pi(0))$ where $\pi\colon \CC^3\to \CC^3/G$ is the quotient map.
   The main result of \cite[Theorem~1.1]{C21} implies  that the classes 
 $[\Psi(\mathcal{O}_0\otimes\rho)]$ for $\rho\neq \rho_0$ provide a basis for $F_2/F_0$, and furthermore, the quotient map 
 \[
 G_C\colon F_2/F_0\longrightarrow F_2/F_1
 \]
sends each class $[\Psi(\mathcal{O}_0\otimes\rho)]$ to a specific linear combination of the basis $[\mathcal{O}_{E_1}], \dots, [\mathcal{O}_{E_\ell}]$ of $F_2/F_1$ associated to the irreducible components $E_1, \dots, E_\ell$ of dimension two in the fibre $\tau^{-1}(\pi(0))$. To state our second main result, we say that a matrix is \emph{sign-coherent} if for each column, the nonzero entries are either all positive or all negative. 

 \begin{theorem}
 \label{thm:signcoherentintro}
 Let $G\subset \SL(3,\CC)$ be a finite subgroup. The matrix defining the map $G_C$ with respect to the bases described above is sign-coherent. 
  \end{theorem}

 The explicit description of $G_C$ that follows from Theorem~\ref{thm:mainintro} (see Proposition~\ref{prop:GaleDual}) allows us to read off the support with multiplicities of the objects $\Psi(\mathcal{O}_0\otimes \rho)$ for $\rho\neq \rho_0$ directly from the $\rho^{\text{th}}$ column of this sign-coherent matrix. The unusual naming convention for the cases in Theorem~\ref{thm:mainintro} is explained by the three possibilities for the sign of the entries in the $\rho^{\text{th}}$ column of this matrix: 
 \begin{enumerate}
     \item[\ensuremath{(+)}] the nonzero entries are positive if and only if $\rho$ satisfies case \ensuremath{(+)} of Theorem~\ref{thm:mainintro};
     \item[(0)] the column is the zero vector if and only if $\rho$ satisfies case \ensuremath{(0)} of Theorem~\ref{thm:mainintro}; or
     \item[($-$)] the nonzero entries are negative if and only if $\rho$ satisfies case \ensuremath{(-)} of Theorem~\ref{thm:mainintro}.
 \end{enumerate}
This trichotomy for nontrivial irreducible representations is a well-known feature of \emph{Reid's recipe} for a finite abelian subgroup of $\SL(3,\CC)$, see \cite[Theorem~1.2]{Logvinenko10} or \cite[Theorem~1.4]{BCQ15}. In Section~\ref{sec:signcoherent}, we illustrate this recipe for nonabelian subgroups  by applying Theorems~\ref{thm:mainintro} and \ref{thm:signcoherentintro} to the examples of order 39 and 24  that were studied first by Nolla de Celis~\cite{NolladeCelis21} (compare \cite[Example~6.5]{C21}).

 \medskip

 Theorem~\ref{thm:signcoherentintro} provides us with the information of the supports of the objects $\Psi(\mathcal{O}_0\otimes \rho)$, so it is possible to write down explicitly the assignment 
 \[
\rho \longmapsto \supp\Psi(\mathcal{O}_{0}\otimes \rho)
  \]
 in the spirit of \eqref{eqn:geometricMcKay}. However, this is not a bijection for a general finite subgroups of $\SL(3,\CC)$; indeed, it was well-known to Cautis and Logvinenko that, for the cyclic group action of type $\frac{1}{6}(1,2,3)$,  the objects $\Psi(\mathcal{O}_0\otimes \rho)$ associated to three of the nontrivial irreducible representations $\rho$ share the same support. Therefore, to understand the McKay correspondence bijection in full, one must go beyond computing the support by calculating the pure sheaves $\Psi(
 \mathcal{O}_0\otimes \rho)$ themselves. This leads one naturally to seek a nonabelian generalisation of the \emph{derived McKay correspondence} \cite{CCL17}.

  We take the first steps in this direction by calculating explicitly the objects  $\Psi(
 \mathcal{O}_0\otimes \rho)$ in case \ensuremath{(0)} of Theorem~\ref{thm:mainintro}. To state the result, note that in case \ensuremath{(0)} of Theorem~\ref{thm:mainintro}, the intersection of the hyperplane $(\theta_\rho=0)$ with the closure of $C$ defines a codimension-one wall $W_\rho$ of $C$. Define $Z_\rho$ to be closed subset of $Y$ obtained as the intersection of the fibre $\tau^{-1}(\pi(0))$ with the strictly-semistable locus for the wall $W_\rho$. Our analysis of the morphism induced by variation of GIT quotient from $C$ into the wall $W_\rho$ allows us to establish the following result (see Proposition~\ref{prop:typeIorIII}).

  \begin{theorem}
    \label{thm:BCQconjectureintro}
 Suppose that $\rho\neq \rho_0$ is such that case \ensuremath{(0)} of Theorem~\ref{thm:mainintro} holds, or equivalently, suppose that the support of $\Psi(\mathcal{O}_0\otimes \rho)$ has dimension one. Then the locus $Z_\rho$ in $Y$ is isomorphic to $\mathbb{P}^1$, and we have $\Psi(\mathcal{O}_0\otimes \rho)\cong \det(\mathcal{R}_\rho)^{-1}\vert_{Z_\rho}$.
 \end{theorem}
 
  This result establishes the conjecture of Bocklandt--Craw--Quintero V\'{e}lez~\cite[Conjecture~5]{BCQ21} when $\rho$ is such that case \ensuremath{(0)} of Theorem~\ref{thm:mainintro} holds. In case \ensuremath{(0)}, the line bundle $\det(\mathcal{R}_\rho)$ has degree one on the curve $\ell:=Z_\rho\cong \mathbb{P}^1$ (see Lemma~\ref{lem:deg0or1}), so Theorem~\ref{thm:BCQconjectureintro} implies that 
  \[
  \Psi(\mathcal{O}_0\otimes \rho)\cong \mathcal{O}_{Z_\rho}(-1)
  \]
  which is the form presented in Theorem~\ref{thm:mainintro}(0). In fact, for every irreducible representation $\sigma\neq \rho$, the restriction of the tautological bundle $\mathcal{R}_\sigma$ on $Y$ to the curve $\ell$ is trivial, while $\mathcal{R}_\rho$ is globally generated of degree one. This reflects precisely the behaviour of the tautological bundles on the $G$-Hilbert scheme for a finite subgroup $G\subset \SL(2,\CC)$ as in \cite[Lemma~2.1]{KV00}. Note that \cite[Conjecture~5]{BCQ21} remains open in case \ensuremath{(+)} of Theorem~\ref{thm:mainintro}.

Finally, the Cautis--Logvinenko conjecture was extended
by the first author \cite{C21} to a statement for threefold noncommutative crepant resolutions (NCCRs) for which the vertex simple modules exist. Here, we treat only the classical case of the conjecture.
 
\smallskip

 \noindent \textbf{Notation}
 We work throughout over the base field $\CC$ of complex numbers. When working with modules over an algebra, we consider left-modules only. For a finitely generated $\ZZ$-module $\Lambda$, we write $\Lambda_{\QQ}:=\Lambda\otimes_{\ZZ} \QQ$. 
 
\smallskip

 \noindent \textbf{Acknowledgements}
 We thank Tom Bridgeland and Matthew Pressland for an inspiring conversation that provided the catalyst for this project.  The first and third authors were supported by Research Project Grant RPG-2021-149 from The Leverhulme Trust. In addition, the first author is also supported by EPSRC grant UKRI2779,
 the second author is supported by
the Yushan Fellow Program of the Ministry of Education in Taiwan 
(MOE-115-YSFMS-0003-002-P2), 
the National Council for Science and Technology in Taiwan (NSTC 115-2628-M-002-009), and the Asian Young Scientist Fellowship,
 and the third author is supported by JSPS KAKENHI Grants numbered JP26K06728 and JP19K14504.

 \section{The $G$-Hilbert scheme}
 We begin by recalling some well-known results about the $G$-Hilbert scheme for a finite subgroup $G\subset \SL(3,\mathbb{C})$.

 \subsection{Skew group algebra}
 Let $G\subset \SL(3,\mathbb{C})$ be a finite subgroup. 
  Let $\Irr(G)$ denote the set of isomorphism classes of irreducible representations of $G$, where $\rho_0$ is the trivial representation, and write $R(G)=\bigoplus_{\rho\in \Irr(G)} \ZZ \rho$ for the representation ring of $G$.   Write $\rho_{\text{given}}$ for the three-dimensional representation of $G$ given by the inclusion of $G\subset \SL(3,\CC)$.

 The skew group algebra $A:= \CC[x,y,z]\rtimes G$ is the $\CC$-algebra whose underlying vector space is $\CC[x,y,z] \otimes_{\CC} \CC[G]$, and where the multiplication is defined first on pure tensors by 
 \[ (f \otimes g)(f' \otimes g') = f\,g(f') \otimes gg', 
 \] for $f,f' \in \CC[x,y,z]$ and $g,g' \in G$, and is then extended \(\CC\)-bilinearly to all of $A$. Since $G\subset \SL(3,\mathbb{C})$,  $A$ is a Calabi--Yau algebra of global dimension three. 
 
 The \emph{McKay quiver} of $G$ is the quiver $Q$ with vertex set $\Irr(G)$, where for each $\rho, \rho'\in \Irr(G)$, there are $\dim_{\mathbb{C}}\Hom_G(\rho', \rho\otimes \rho_{\text{given}})$ arrows from vertex $\rho$ to vertex $\rho'$.  It is well-known that the skew group algebra $A$ is Morita equivalent to a quotient $A^\flat$ of the path algebra of $Q$ modulo a two-sided ideal of relations \cite[Theorem~2.1]{GM02}, so the category of $A$-modules is equivalent to the category of $A^\flat$-modules, and hence also to the category of representations of $Q$ that satisfy the relations.  For each $\rho\in \Irr(G)$, define the $A$-module 
 \[
 S_\rho:= \CC\otimes \rho,
 \]
 where the $A$-module structure satisfies $(f\otimes g)\cdot v=f(0)\rho(g)v$ for $f\in \CC[x,y,z]$, $g\in G$, and $v\in \rho$. Observe that every such $A$-module $S_\rho$ is simple because $\rho$ is irreducible. Every $A$-module $V$ is in particular a $\CC[G]$-module, so it decomposes  
 \begin{equation}
     \label{eqn:Gmodule}
  V=\bigoplus_{\rho\in \Irr(G)}V_\rho\otimes \rho
 \end{equation}
 as a representation of $G$. The \emph{dimension vector} of a finite dimensional $A$-module $V$ is the vector $d(V):=(\dim V_\rho)_{\rho\in \Irr(G)}\in R(G)$.

 The global sections functor gives an equivalence between the abelian category of $G$-equivariant coherent sheaves on $\mathbb{C}^3$ and the abelian category of finitely generated $A$-modules. Recall that a finite-dimensional $A$-module $V$ is \emph{nilpotent} if there exists $n\in \mathbb{N}$ such that any path in the McKay quiver of length greater than $n$ acts on the $A^\flat$-module corresponding to $V$ as the zero map. 
 
 \begin{lemma}
 \label{lem:abelianequiv}
    The global sections functor identifies the abelian category of $G$-equivariant coherent sheaves on $\mathbb{C}^3$ with support at $0\in \mathbb{C}^3$ and the abelian category of nilpotent left $A$-modules. For $\rho\in \Irr(G)$, this equivalence identifies $\mathcal{O}_0\otimes \rho$ with the simple $A$-module $S_\rho=\CC\otimes \rho$.
 \end{lemma}
 \begin{proof}
  Following \cite[Proof of Lemma~8.1]{BKR01}, the simple $G$-equivariant coherent sheaves supported at $0\in \mathbb{C}^3$ are precisely $\{\mathcal{O}_0\otimes \rho\}_{\rho\in \Irr(G)}$, and every such sheaf admits a Jordan--Holder filtration having these simple objects as composition factors. Since $S_\rho\cong H^0(\mathcal{O}_0\otimes \rho)$, the corresponding $A$-modules admit filtrations whose factors are the simple $A$-modules $\{S_\rho \}_{\rho\in \Irr(G)}$. The result follows from the Morita equivalence between $A$ and $A^\flat$ because an $A^\flat$-module admits such a filtration if and only if it is nilpotent.
 \end{proof}

 Write $K(A)$ and $K_{\text{fin}}(A)$ for the Grothendieck groups of the categories of finitely-generated $A$-modules and finite-dimensional $A$-modules respectively. The Euler form determines a bilinear pairing $\chi_A\colon K(A)\times K_{\text{fin}}(A)\to \ZZ$ given by
\[
\chi_A(E,F)=\sum_{i\in \ZZ} (-1)^i \dim_{\CC} \Ext_A^i(E,F).
\]
 The numerical Grothendieck group $\Knum(A)$ of finitely-dimensional $A$-modules is defined to be the quotient $\Knum(A):= K_{\text{fin}}(A)/K(A)^\perp$. For any finite-dimensional $A$-module $V$, the decomposition \eqref{eqn:Gmodule} shows that the class of $V$ satisfies $[V]=\sum_{\rho\in \Irr(G)} (\dim V_\rho) [S_\rho]\in \Knum(A)$, and the map
 \begin{equation}
     \label{eqn:KnumA}
     d\colon \Knum(A)\longrightarrow R(G)
 \end{equation}
 sending the class of $V$ to its dimension vector $d(V)$ is a $\ZZ$-linear isomorphism (injectivity follows as in the proof of \cite[Lemma~7.1.1]{BCZ17}).

 \subsection{The $G$-Hilbert scheme}
   The $G$-Hilbert scheme, denoted $\ghilb(\CC^3)$, is defined to be the fine moduli space of \emph{$G$-clusters}, that is, $G$-invariant subschemes of $\CC^3$ for which $H^0(\mathcal{O}_Z)$ is isomorphic as a $\CC[G]$-module to $\CC[G]$. For every $G$-cluster $Z$, the sheaf $\mathcal{O}_Z$ is a $G$-equivariant coherent sheaf, and there is a surjective $A$-module homomorphism $A \to H^0(\mathcal{O}_Z)$ making $H^0(\mathcal{O}_Z)$ a cyclic $A$-module that is generated by the image of $1\otimes e\in A$, where $e=\frac{1}{\vert G\vert} \sum_{g\in G} g\in \CC[G]$ is the idempotent corresponding to the trivial representation. Conversely,  an $A$-module $V$ is the space of sections of a $G$-cluster if there is a surjective $A$-module homomorphism $Ae \to V$ satisfying $\dim V_\rho=\dim\rho$ for each $\rho\in \Irr(G)$.

   There is a Hilbert--Chow morphism 
  \[
  \tau\colon Y:= \ghilb(\mathbb{C}^3)\longrightarrow \mathbb{C}^3/G
  \]
  that sends each $G$-cluster to the $G$-orbit that supports it. Let $\mathcal{Z}\subseteq Y\times \CC^3$ denote the universal subscheme, and write $p$ and $q$ for the restriction to $\mathcal{Z}$ of the first and second projections respectively. Let $\pi\colon \CC^3\to \CC^3/G$ denote the quotient morphism.  There is a commutative diagram of schemes
      \begin{equation}
 \label{eqn:BKRdiagram}
\begin{tikzcd}
  & \mathcal{Z}\ar[dl,swap,"p"] \ar[dr,"q"]& \\
Y\ar[dr,swap,"\tau"] & & \CC^3\ar[dl,"\pi"] \\
  & \mathbb{C}^3/G & 
  \end{tikzcd}
  \end{equation}
  where $q$ and $\tau$ are birational, and where $p$ is finite and flat. The sheaf $\mathcal{R}:=p_*\mathcal{O}_{\mathcal{Z}}$ is locally free of rank $|G|$ and admits a decomposition $\mathcal{R}\cong\bigoplus_{\rho\in \Irr(G)} \mathcal{R}_\rho\otimes \rho$ into indecomposable summands,
   where $\mathcal{R}_\rho$ is locally free of rank $\dim \rho$. 
   
 Let $D(Y)$ denote the bounded derived category of coherent sheaves on $Y$, and let $D(A)$ denote the bounded derived category of finitely-generated left $A$-modules. Identify the abelian category of $G$-equivariant coherent sheaves with the category of finitely generated left $A$-modules. We write
 \begin{equation}
  \label{eqn:BKR}
 \Phi(-):=\Rderived q_*p^*(-)\colon D(Y)\longrightarrow D(A),
   \end{equation}
 for the Fourier--Mukai transform with kernel $\mathcal{O}_{\mathcal{Z}}$ that was studied by Bridgeland--King--Reid~\cite{BKR01}.
    
\begin{theorem}[Bridgeland--King--Reid]
\label{thm:BKR}
The Hilbert--Chow morphism $\tau\colon Y\to \mathbb{C}^3/G$ is a projective crepant resolution, and the functor $\Phi$ from \eqref{eqn:BKR} is an equivalence of derived categories. In addition, the restriction of $\Phi$ to the full subcategory $D_{0}(Y)$ of objects supported on $\tau^{-1}(\pi(0))$ is an equivalence
 \[
 \Phi_0\colon D_0(Y)\longrightarrow \Dnil(A),
 \] where 
 $\Dnil(A)$ is the full subcategory of objects comprising nilpotent $A$-modules.
\end{theorem}
 \begin{proof}
 The results of the first sentence are explicit in \cite{BKR01}. In addition, the restriction $\Phi_0$ of the Fourier--Mukai transform from \eqref{eqn:BKR} to $D_0(Y)$ gives an equivalence between $D_0(Y)$ and the full subcategory $D^G_0(\mathbb{C}^3)$ of the bounded derived category of $G$-equivariant coherent sheaves on $\mathbb{C}^3$ consisting of objects supported at the origin in $\mathbb{C}^3$ \cite[Section~6, Step~7]{BKR01}. The equivalence of abelian categories from Lemma~\ref{lem:abelianequiv} allows us to identify $D^G_0(\mathbb{C}^3)$ with $\Dnil(A)$ as triangulated categories. 
 \end{proof}

  For each $y\in Y$, the equivalence $\Phi$ sends the skyscraper sheaf $\mathcal{O}_y$ to the $G$-equivariant coherent sheaf $\mathcal{R}\vert_y$ obtained as the fibre of the tautological bundle $\mathcal{R}$ on $Y=\ghilb(\CC^3)$ over the point $y$, and we identify this in turn with the $A$-module \[
  V_y:=H^0(\mathcal{R}\vert_y)
  \]
  whose dimension vector is equal to the regular representation of $G$.
 
\begin{corollary}
\label{cor:fibrenilpotent}
 Let $y\in Y$. Then $y\in \tau^{-1}(\pi(0))$ if and only if the $A$-module $V_y$ is nilpotent.
 \end{corollary}
 \begin{proof}
 We have $y\in \tau^{-1}(\pi(0))$ if and only if $\supp(\mathcal{O}_y)\subseteq \tau^{-1}(\pi(0))$. Theorem~\ref{thm:BKR} shows that this holds if and only if $V_y=\Phi(\mathcal{O}_y)$ is nilpotent.
 \end{proof}
 
 Of particular interest to us is the functor $\Psi\colon D(A)\rightarrow D(Y)$ that is quasi-inverse to $\Phi$.
 
\begin{corollary}
\label{cor:PsiSrhofibre}  For each $\rho\in \Irr(G)$, the support of the object $\Psi(S_\rho)$ is contained in $\tau^{-1}(\pi(0))$.
\end{corollary}

 Finally, we record a useful result that builds on work of the first author with Ishii~\cite{CI04}:

 \begin{lemma}
     \label{lem:curveintauinverse0}
 Each extremal ray of the Mori cone of curves $\overline{\operatorname{NE}}(Y)$ is spanned by the class of a proper curve $\ell$ that lies in the locus $\tau^{-1}(\pi(0))$.    
     \end{lemma}
     \begin{proof}
    Let $C \subset Y$ be a proper curve which spans an extremal ray of $\overline{\operatorname{NE}}(Y)$.
    Since $\mathbb{C}^3/G$ is affine, $t := \tau(C)$ is a point. We're done if $t=\pi(0)$, so we suppose otherwise. 
    
    As $\tau \colon Y \to \mathbb{C}^3/G$ is $\mathbb{C}^*$-equivariant, we obtain a flat family
    $\lambda \cdot C \subset Y$ of proper curves parameterised by $\lambda \in \mathbb{C}^*$
    representing the same class 
    $[C] \in \overline{\operatorname{NE}}(Y)$.
    The morphism $\mathbb{C}^* \to \mathbb{C}^3/G$ sending $\lambda$ to
    $\tau(\lambda \cdot C) = \lambda \cdot t$ allows us to regard $\mathbb{C}^*$ as a scheme over  $\mathbb{C}^3/G$. We then have a moduli map
    \[
    \phi \colon \mathbb{C}^* \to \mathrm{Hilb}(\tau)
    \]
    over $\mathbb{C}^3/G$
    to the relative Hilbert scheme 
    $\mathrm{Hilb}(\tau)$ parameterising proper curves in the fibres of 
    $\tau \colon Y \to \mathbb{C}^3/G$.

    Let $H \subset \mathrm{Hilb}(\tau)$ be an irreducible component
    containing the image of $\phi$
    and let $\sigma \colon H \to \mathbb{C}^3/G$ be the structural morphism.
    Note that $\mathbb{C}^* \to \mathbb{C}^3/G$ extends to a morphism
    $\alpha \colon \mathbb{C} \to \mathbb{C}^3/G$ with $\alpha(0) = \pi(0)$.
    As $\tau$ is projective,  
    $\sigma$ is proper~\cite[Th\'eor\`eme 3.2]{GroHilb} (see also \cite[Lemma 0DPH]{Stacks}), so the morphism $\phi\colon \mathbb{C}^* \to H$ 
    extends to a morphism
    \[
    \widetilde{\phi} \colon\mathbb{C} \longrightarrow H
    \]over $\mathbb{C}^3/G$.
    We have 
    $\sigma(\widetilde{\phi}(0)) = \alpha(0) = \pi(0)$,
    so the proper curve $\ell \subset Y$ that 
    $\widetilde{\phi}(0)$ parameterises 
    lies in $\tau^{-1}(\pi(0))$. 
    As $\ell$ is a flat limit of the curves $\lambda \cdot C$ as 
    $\lambda \to 0$, it has the same curve class as $C$.
\end{proof}

 \subsection{Quiver moduli space description}
 The $G$-Hilbert scheme may be regarded as a fine moduli space of $A$-modules
 \cite{ItoNakajima00,Crawthesis}. 
 
 To see this, let $v:=\sum_\rho (\dim\rho)[S_\rho]\in \Knum(A)$ denote the vector that's identified with the regular representation of $G$ under the isomorphism  
 \eqref{eqn:KnumA}, and consider the rational vector space 
   \[
   \Theta:= \big\{\theta\in \Hom_{\ZZ}(\Knum(A), \QQ) \mid \theta(v)=0\big\}.
   \]
   For $\theta\in \Theta$ and $\rho\in \Irr(G)$, write $\theta_\rho:= \theta([S_\rho])$ for the component of $\theta$ indexed by $\rho$. Each finite-dimensional $A$-module $V$ determines the class $[V]=\sum_\rho (\dim V_\rho) [S_\rho]$ in $\Knum(A)$, and for $\theta\in \Theta$, we set
   \[
   \theta(V):=\theta([V]) = \sum_{\rho\in \Irr(G)}  (\dim V_\rho)\theta_\rho.\]
   An $A$-module $V$ of dimension vector $v$ is \emph{$\theta$-semistable} (resp.\ \emph{$\theta$-stable}) if  $\theta(V')\geq 0$ (resp.\ $\theta(V')>0$) for every nonzero, proper $A$-submodule $V'$ of $V$. We say $\theta$ is \emph{generic} if every $\theta$-semistable $A$-module of dimension vector $v$ is $\theta$-stable.

  We recalled earlier that the skew group algebra $A$ is Morita equivalent to a quotient of the path algebra of the McKay quiver $Q$ modulo a two-sided ideal of relations, so the category of $A$-modules is equivalent to the category of representations of $Q$ that satisfy the relations. As a result, for any $\theta\in \Theta$, work of King~\cite{King94} allows us to construct the coarse moduli space $\mathcal{M}_\theta(A,v)$ of $\theta$-semistable $A$-modules of dimension vector $v$ (up to S-equivalence) as a GIT quotient. For generic $\theta$, this GIT quotient $\mathcal{M}_\theta(A,v)$ is the fine moduli space of $\theta$-stable $A$-modules of dimension vector $v$ (up to isomorphism), in which case, $\mathcal{M}_\theta(A,v)$ carries a tautological locally-free sheaf $T=\bigoplus_{\rho\in \Irr(G)} T_\rho$ such that $\operatorname{rank} T_\rho=\dim \rho$ for all $\rho\in \Irr(G)$, with $T_{\rho_0}\cong \mathcal{O}_{\mathcal{M}_\theta(A,v)}$, together with a tautological $\CC$-algebra homomorphism $A\to \End(T)$.

The rational vector space $\Theta$ admits a finite, polyhedral wall-and-chamber decomposition in which GIT stability parameters $\theta, \theta'\in \Theta$ lie in the same GIT cone if and only if the notions of $\theta$-stability and $\theta'$-stability coincide. The union of the interiors of the top-dimensional cones is precisely the set of all  generic GIT parameters in $\Theta$. Every stability condition that is contained in the positive orthant $\Theta_+=\big\{\theta\in \Theta \mid \theta_\rho > 0 \text{ for }\rho\neq \rho_0\big\}$ is generic, so $\Theta_+$ is contained in a GIT chamber denoted $C$. For $\theta\in C$, it is well known (see, for example, \cite[Proposition~5.9]{Crawthesis}) that there is an isomorphism  
\[
\mathcal{M}_\theta(A,v)\cong \ghilb(\mathbb{C}^3)
\]
 that identifies the locally-free sheaves $\mathcal{R}_\rho\cong T_\rho$ for all $\rho\in \Irr(G)$. These isomorphisms identify the tautological sheaf $T$ on $\mathcal{M}_\theta(A,v)$ with a subbundle of the tautological sheaf $\mathcal{R}$ on $\ghilb(\mathbb{C}^3)
$ in which each indecomposable summand appears precisely once. In particular,  \cite[Corollary~2.4]{CIK18} implies that the sheaf $\mathcal{R}_\rho$ is globally generated for each $\rho \in \Irr(G)$.

\begin{lemma}
\label{lem:C}
  Let $\rho\neq \rho_0$. Either the hyperplane $(\theta_\rho=0)$ in $\Theta$ intersects the chamber $C$, or it is a supporting hyperplane that contains a codimension-one face of the closure $\overline{C}$.   
\end{lemma}
\begin{proof}
The hyperplane $(\theta_\rho=0)$ contains a codimension-one face of the closure $\overline{\Theta_+}$. If there exists $\theta\in C$ satisfying $\theta_\rho<0$, then $(\theta_\rho=0)$ intersects the (open) chamber $C$. Otherwise, $C$ lies in the half-space $(\theta_\rho>0)$, and hence $\overline{\Theta_+}\cap (\theta_\rho=0)$ is contained in $\overline{C}\cap (\theta_\rho=0)$. The former is of codimension-one, hence so is the latter.
\end{proof}

   \section{The first case of the trichotomy}
 We begin by highlighting the importance of the chamber $C$. Recall that for each $y\in Y$, we write $V_y$ for the $A$-module obtained as the fibre over $y$ of the tautological bundle $\mathcal{R}$ on
 $Y=\ghilb(\CC^3)$. 

 \begin{lemma}
 \label{lem:HomVySrho}
 For $y\in Y$, we have $\Hom_A(V_y,S_\rho)=0$ for $\rho\neq \rho_0$.
 \end{lemma}
 \begin{proof}
 For $\theta\in C$, the kernel $K$ of any nonzero homomorphism $V_y\to S_\rho$ would satisfy $\theta(K)=\theta(V_y) - \theta(S_\rho) <0$ because $\rho\neq\rho_0$. However, the $A$-module $V_y$ is $\theta$-stable, a contradiction.
   \end{proof}
 
 To compute the cohomology sheaves of the object $\Psi(S_\rho)$ in $D(Y)$,  we begin by taking a projective resolution of $S_\rho$ as a left $A$-module
 and then we apply $\Psi$. The global dimension of $A$ is three, so $\Psi(S_\rho)$ is quasi-isomorphic to a complex of locally-free sheaves on $Y$ of the form
 \[
 P^{-3}\longrightarrow P^{-2}\longrightarrow P^{-1}\longrightarrow P^0.
 \]
 In fact, the following stronger statement holds.
 
 \begin{lemma}
 \label{lem:Hom3}
 For $\rho\neq \rho_0$, the object $\Psi(S_\rho)$ is quasi-isomorphic to a complex of locally-free sheaves 
 \[
 P^{-2}\longrightarrow P^{-1}\longrightarrow P^0.
 \]
 Moreover, each irreducible component of the support of $\Psi(S_\rho)$ has codimension at most two.
  \end{lemma}
  \begin{proof}
  Let $y\in Y$ and $\rho\neq \rho_0$. Since $A$ is a CY3 algebra, we obtain
 \begin{equation}
 \label{eqn:Hom3}
 \Hom_{D(Y)}^3(\Psi(S_\rho), \mathcal{O}_y) \cong \Ext^3_A(S_\rho, V_y) \cong \Hom_A(V_y,S_\rho)^* = 0
 \end{equation}
 by Lemma~\ref{lem:HomVySrho}. Since $Y$ is quasi-projective, we may apply the intersection theorem from Bridgeland--Maciocia~\cite[Proposition~5.4]{BM02} to obtain the first statement. In particular, $\Psi(S_\rho)$ has homological dimension at most two. The second statement now follows from \cite[Corollary~5.5]{BM02}.
 \end{proof}
 
 We are already in a position to prove Conjecture~\ref{conj:CLintro} for any irreducible representation $\rho\neq \rho_0$ such that $(\theta_\rho=0)$ is not a supporting hyperplane of the chamber $C$. 
   
 \begin{proposition}
 \label{prop:type1}
 For $\rho\neq \rho_0$, suppose that $(\theta_\rho=0)$ is not a supporting hyperplane of $C$. Then 
 \[
 H^k\big(\Psi(S_\rho)\big)=0 \quad\text{for }k\neq -1.
 \]
 Moreover, every irreducible component of the support of $\Psi(S_\rho)$ is of codimension one.
 \end{proposition}
 
\begin{proof}
 Lemma~\ref{lem:C} implies that there exists $\theta\in C$ satisfying $\theta_\rho=0$, so $S_\rho$ is not isomorphic to a submodule of $V_y$ for any $y\in Y$. Since $S_\rho$ is simple, it follows that 
 \[
 0=\Hom_A(S_\rho,V_y) \cong \Hom_{D(Y)}(\Psi(S_\rho),\mathcal{O}_y) \quad \text{for all }y\in Y.
 \]
 Applying the same logic as in Lemma~\ref{lem:Hom3} shows that  $\Psi(S_\rho)$ is quasi-isomorphic to a complex of locally-free sheaves of the form
 \[
 P^{-2}\stackrel{f}{\longrightarrow} P^{-1}.
 \]
 The object $\Psi(S_\rho)$ has proper support because it is a closed subset of the proper fibre $\tau^{-1}(\pi(0))$ by Corollary~\ref{cor:PsiSrhofibre}. so $\ker(f)=0$. This proves the first statement, while the statement about the codimension follows from \cite[Corollary~5.5]{BM02}. 
\end{proof}

 \section{The unstable locus in the fibre over the origin}
  It remains to prove Conjecture~\ref{conj:CLintro} for $\rho\neq \rho_0$ such that the hyperplane $(\theta_\rho=0)$ in $\Theta$ is a supporting hyperplane of the closure of the GIT chamber $C$. This means that the closed cone
 \[
 W_\rho:= \overline{C}\cap (\theta_\rho=0)
 \]
 is a codimension-one face of the closure of $C$; we say that  $W_\rho$ is a \emph{GIT wall} of $C$. Each stability condition $\eta\in \Theta$ that lies in the relative interior of $W_\rho$ satisfies $\eta_\rho=0$ and $\eta_\sigma>0$ for all $\sigma\not\in\{\rho_0,\rho\}$.   For any such $\eta\in \relint(W_\rho)$, define the closed set
  \begin{equation}
     \label{eqn:Zrho} 
 Z_\rho := \{ y\in Y \mid V_y\text{ is a strictly }\eta\text{-semistable}\}\cap \tau^{-1}\big(\pi(0)\big)
  \end{equation}
  in $Y$. Note that the locus $\{ y\in Y \mid V_y\text{ is a strictly }\eta\text{-semistable}\}$ is often contained in $\tau^{-1}(\pi(0))$, but this is not always the case, see \cite[Example~1]{BCQ21}. In particular, we cannot omit taking the intersection with $\tau^{-1}\big(\pi(0)\big)$ in \eqref{eqn:Zrho}.

 \begin{lemma}
 \label{lem:Homnonzero}
 The locus $Z_\rho$ coincides with the locus of points $y\in Y$ such that $V_y$ is nilpotent and strictly $\eta$-semistable. In particular,   for each $y\in Z_\rho$, there is a short exact sequence
 \begin{equation}
 \label{eqn:sesV_y}
 0\longrightarrow S_\rho
 \longrightarrow V_y\longrightarrow Q_y\longrightarrow 0
 \end{equation}
 of $\eta$-semistable $A$-modules.\end{lemma}
 
 \begin{proof}
 
 The first sentence is immediate from Corollary~\ref{cor:fibrenilpotent}. 
 For the second, it is enough by the Morita equivalence between $A$ and $A^\flat$ to prove the statement in the category of $A^\flat$-modules. 
 The $A^\flat$-module
$V_y$ is strictly $\eta$-semistable, so there is a destabilising submodule $S_y\subset V_y$ with $\eta(S_y)=0$, and $V_y$ is nilpotent, so the submodule $S_y$ is nilpotent. It follows that $S_y$ admits a Jordan--Holder filtration whose composition factors are copies of the simple $A^\flat$-module $S_\rho$. In particular,  $S_\rho$ is a submodule of $S_y$ and hence of $V_y$. We have that $\eta(S_\rho)=0$ and $V_z$ is $\eta$-semistable as $z\in Z_\rho$, so $Q_y:= V_y/S_\rho$ is also $\eta$-semistable.
 \end{proof}

\begin{lemma}
\label{lem:h}
 Let $y, z\in Z_\rho$. If $\Hom(V_y,Q_z)\neq 0$, then there is a short exact sequence of $A$-modules
    \begin{equation}
 \label{eqn:sesV_yQ_z}
 0\longrightarrow S_\rho
 \longrightarrow V_y\longrightarrow Q_z\longrightarrow 0.
 \end{equation}
\end{lemma}
\begin{proof}
 Let $h\colon V_y\to Q_z$ be nonzero. There exists a surjective map of $A$-modules $p\colon Ae\to V_y$, and we write $v_{\rho_0}:= p(1\otimes e) \in V_y$.  Let $v\in V_y$ satisfy $h(v)\neq 0$, and write $v=av_{\rho_0}$ for some $a\in A$. Then $h(v)= ah(v_{\rho_0})$,  so $h(v_{\rho_0})\neq 0$.  But $ \dim (Q_z)_{\rho_0}=1$ and $Q_z$ is a cyclic $A$-module with generator in $(Q_z)_{\rho_0}$, so $h$ is surjective. The dimension vectors of $V_y$ and $V_z$ agree, so $N:=\ker(h)$ satisfies 
 \[
 [N]=[V_y]-[Q_z] = [V_z]-[Q_z]=[S_\rho]\in \Knum(A)
 \]
 by the sequence \eqref{eqn:sesV_y} associated to the point $z$. It remains to show that $N\cong S_\rho$. 
 
It is enough by the Morita equivalence between $A$ and $A^\flat$ to prove $N\cong S_\rho$ 
 in the category of $A^\flat$-modules.
 For this, $V_y$ is
nilpotent by Lemma~\ref{lem:Homnonzero}, hence so is the submodule $N$. Since $\dim N_\rho=1$, any loop in the McKay quiver at vertex $\rho$ acts on $N$ as a scalar, and this action must be zero as $N$ is nilpotent.
\end{proof}

 The support of the object $\Psi(S_\rho)$ is the union of the supports of the cohomology sheaves of $\Psi(S_\rho)$.  Since only finitely many of these sheaves are nonzero, it follows that $\supp \Psi(S_\rho)$ is closed. The proof of the next result is adapted from \cite[Section~11]{IU16}.

 \begin{proposition}
 \label{prop:support}
 The object $\Psi(S_\rho)$ has proper support, and in fact $\supp \Psi(S_\rho) = Z_\rho$.
 \end{proposition}
 \begin{proof}
The support of $\Psi(S_\rho)$ is a closed subset of the proper fibre $\tau^{-1}(\pi(0))$ by Corollary~\ref{cor:PsiSrhofibre}, so it is proper. To determine the support explicitly, 
 let $y\in Z_\rho$. Lemma~\ref{lem:Homnonzero} implies that 
\[
\Hom_{D(Y)}(\Psi(S_\rho),\mathcal{O}_y) \cong \Hom_A(S_\rho,V_y)\neq 0,
\]
 so $y\in \supp \Psi(S_\rho)$. For the opposite inclusion, suppose $y\not\in Z_\rho$. If $y\not\in \tau^{-1}(\pi(0))$, then taking complements in the inclusion $
 \supp(\Psi(S_\rho))\subseteq \tau^{-1}(\pi(0))$ gives $y\not\in  \supp\Psi(S_\rho)$ as required. Otherwise,  $y$ lies in $\tau^{-1}(\pi(0))$ and $V_y$ is not strictly $\eta$-semistable. To show that $y\not \in \supp \Psi(S_\rho)$, we require
 \[
 0=\Hom_{D(Y)}^k(\Psi(S_\rho),\mathcal{O}_y) \cong \Ext^k_A(S_\rho,V_y)
 \] 
 for all $k\in \ZZ$.  We already know this for $k\not\in \{0,1,2\}$ by Lemma~\ref{lem:Hom3}. For $k=0$, if $\Hom_A(S_\rho,V_y)\neq 0$, then $V_y$ would be strictly $\eta$-semistable which is absurd.  For $k=2$, choose any point $z\in Z_\rho$. Now Lemma~\ref{lem:Homnonzero} gives a short exact sequence  
 \begin{equation}
 \label{eqn:sesV_z}
 0\longrightarrow S_\rho
 \longrightarrow V_z\longrightarrow Q_z\longrightarrow 0.
 \end{equation}
 If $\Hom_A(V_y,Q_z) \neq 0$, then Lemma~\ref{lem:h} determines a submodule $N\cong S_\rho$ of $V_y$ satisfying $\eta(N)>0$, so $N$ destabilises $V_y$, a contradiction. Therefore $\Hom_A(V_y,Q_z)=0$. Since $y\neq z$, we have 
 \begin{equation}
 \label{eqn:ExtBKR}
 \Ext^1_A(V_y,V_z) \cong \Ext^1_{\mathcal{O}_Y}(\mathcal{O}_y,\mathcal{O}_z) =0,
 \end{equation}
 so the long exact sequence associated to $\Hom_A(V_y,-)$ implies that $\Ext^1_A(V_y,S_\rho)=0$. Serre duality gives $\Ext^2_A(S_\rho,V_y)\cong \Ext^1_A(V_y,S_\rho)^*=0$.  Since $\Ext_A^k(S_\rho,V_y)=0$ for all $k\neq 1$, we have 
 \[
 \dim\Ext_A^1(S_\rho,V_y) = -\chi_A(S_\rho,V_y) = -\chi_{\mathcal{O}_Y}\big(\Psi(S_\rho),\mathcal{O}_y\big).
 \]
 Since $Y$ is connected, the value of the Euler characteristic is unchanged if we replace $y$ by any other point of $Y$, and in particular, by any point in the complement of $\supp \Psi(S_\rho)$, for which this Euler characteristic equals zero. Therefore $\Ext^1_A(S_\rho,V_y)=0$, so $y\not\in \supp\Psi(S_\rho)$ after all.  \end{proof}

 Proposition~\ref{prop:support} and Lemma~\ref{lem:Hom3} together imply that each irreducible component of $Z_\rho$ has dimension at least 1. For any point $z\in Z_\rho$, Lemma~\ref{lem:Homnonzero} allows us to construct a destabilising sequence
  \begin{equation}
 \label{eqn:sesV_zdestab}
 0\longrightarrow S_\rho
 \longrightarrow V_z\longrightarrow Q_z\longrightarrow 0
 \end{equation}
of $\eta$-semistable $A$-modules. 
 
 \begin{corollary}
 \label{cor:Lin}
 If points $y, z\in Z_\rho$ satisfy $\Hom_A(V_y,Q_z) \neq 0$, then the $\eta$-semistable $A$-modules $V_y$ and $V_z$ are $\Seshadri$-equivalent.
 \end{corollary}
 \begin{proof}
 The result follows by comparing the short exact sequences  \eqref{eqn:sesV_yQ_z} and \eqref{eqn:sesV_zdestab}.
 \end{proof}

\section{Contracting nonsingular rational curves}
  We now study morphisms from $Y$ induced by variation of GIT quotient, where we vary a stability parameter from the GIT chamber $C$ into a wall of the form $W_\rho$.

 As before, any stability parameter $\eta\in \relint(W_\rho)$ satisfies $\eta_\rho=0$ and $\eta_\sigma> 0$ for each $\sigma\not\in \{\rho_0,\rho\}$. The coarse moduli space $\mathcal{M}_\eta(A,v)$ need not be irreducible in general, but if we let $Y_\eta$ denote the normalisation of the image of the morphism $Y\to \mathcal{M}_\eta(A,v)$ induced by variation of GIT quotient from $\theta\leadsto \eta$, then Zariski's Main Theorem shows that the induced morphism
 \begin{equation}
     \label{eqn:contWrho}
\cont_{W_\rho}\colon Y\longrightarrow Y_{\eta}
 \end{equation}
 has connected fibres. In particular, $\cont_{W_\rho}$ is either an isomorphism or it contracts a curve class. 
 Distinct points $y, z\in Y$ are mapped to the same point under the VGIT morphism $Y\to \mathcal{M}_\eta(A,v)$ if and only if the $G$-clusters $V_y, V_z$ are $\Seshadri$-equivalent in the category of $\eta$-semistable $A$-modules.
 
 Choosing $\eta$ so that $\eta_\sigma\in \ZZ$ for all $\sigma\in \Irr(G)$ ensures that
 \begin{equation}
 \label{eqn:LC+eta}
 L_{C}(\eta) = \bigotimes_{\sigma\neq \rho_0} \det \big(\mathcal{R}_\sigma\big)^{\otimes \eta_\sigma}
 \end{equation}
is a line bundle on $Y$. The choice of chamber $C$ implies that each locally-free sheaf $\mathcal{R}_\sigma$ is globally generated \cite[Proposition~2.4]{CIK18}, 
so the line bundle $L_{C}(\eta)$ is also globally generated. Since $L_C(\eta)$ is obtained by descent from an equivariant line bundle that also determines the polarising ample line bundle on the GIT quotient $\mathcal{M}_\eta(A,v)$ (see, for example, \cite[Lemma~3.3]{CI04}), we have that 
\[
L_{C}(\eta) = {\cont_{W_\rho}}^*(\mathcal{O}_{Y_\eta}(1)).
\]
In particular, the morphism $\cont_{W_\rho}$ is determined by $L_C(\eta)$, in the sense that $\cont_{W_\rho}$ contracts a curve $\ell\subseteq Y$ if and only if $\deg L_C(\eta)\vert_\ell = 0$.

 Recall from \eqref{eqn:Zrho} that $Z_\rho\subseteq Y$ is the intersection of the $\eta$-strictly semistable locus with $\tau^{-1}(\pi(0))$.
\begin{lemma}
\label{lem:deg0or1}
 Suppose that the locus $Z_\rho$ contains a curve $\ell\cong \mathbb{P}^1$ that is contracted by the morphism $\cont_{W_\rho}$. Then for all $\sigma\in \Irr(G)$, the restriction of the tautological bundle $\mathcal{R}_\sigma$ to $\ell$ satisfies 
\[
\mathcal{R}_\sigma\vert_{\ell}\cong \left\{\begin{array}{cr}\mathcal{O}_{\mathbb{P}^1}^{\oplus \dim(\sigma)} & \text{ for }\sigma\neq \rho, \\ 
\mathcal{O}_{\mathbb{P}^1}(1)\oplus \mathcal{O}_{\mathbb{P}^1}^{\oplus \dim(\rho)-1} & \text{ for } \sigma=\rho.
\end{array}\right.
\]
\end{lemma}
\begin{proof}
 Since $v_\sigma:=\dim(\sigma) = \rank(\mathcal{R}_\sigma)$, Grothendieck's decomposition theorem 
 for locally-free sheaves on $\mathbb{P}^1$ 
 gives integers $d_{\sigma,1}\geq d_{\sigma,2}\geq \dots \geq d_{\sigma,v_\sigma}$ such that 
  \begin{equation}
 \label{eqn:Tirestricted}
 \mathcal{R}_\sigma\vert_{\ell}\cong \bigoplus_{1\leq k\leq \dim(\sigma)} \mathcal{O}_{\mathbb{P}^1}(d_{\sigma,k}). 
 \end{equation}
 Each sheaf $\mathcal{R}_\sigma$ is globally generated \cite[Proposition~2.4]{CIK18}, hence so is $\mathcal{R}_\sigma\vert_\ell$. It follows that $d_{\sigma,v_\sigma}\geq 0$, otherwise the composition of the surjective map $H^0(\mathcal{R}_\sigma\vert_\ell)\otimes {\mathcal{O}_\mathbb{P}^1}\to \mathcal{R}_\sigma\vert_\ell$ with the projection to $\mathcal{O}_{\mathbb{P}^1}(d_{\sigma,v_\sigma})$ would be surjective, which is absurd. The line bundle $L_{C}(\eta)$ has degree zero on $\ell$, and we compute by \eqref{eqn:LC+eta} that
 \[
 0 = \deg L_{C}(\eta)\vert_{\ell} = \sum_{\sigma\neq \rho_0} \eta_\sigma \deg \det (\mathcal{R}_\sigma)\vert_{\ell} = \sum_{\sigma\neq \rho_0} \eta_\sigma (d_{\sigma,1}+\dots +d_{\sigma,v_\sigma}).
 \]
 Since $\eta_\rho=0$, the sum runs only over those $\sigma\in \Irr(G)$ satisfying $\sigma\not\in \{\rho_0, \rho\}$. Moreover, we have $\eta_\sigma>0$ for $\sigma\not\in \{\rho_0,\rho\}$ and $d_{\sigma,k}\geq 0$ for $1\leq k\leq v_\sigma$, so $d_{\sigma,k}=0$ for all $\sigma\neq \rho$ and all $1\leq k\leq v_\sigma$. It follows that $\mathcal{R}_\sigma\vert_{\ell}\cong \mathcal{O}_{\mathbb{P}^1}^{\oplus v_\sigma}$ when $\sigma\neq \rho$. 
 
 It remains to consider the case $\sigma=\rho$. The previous paragraph implies that $\det(\mathcal{R}_\rho\vert_\ell)$ is globally generated, so $\det(\mathcal{R}_\rho\vert_{\ell}) \cong \mathcal{O}_{\mathbb{P}^1}(d)$ for $d:=d_{\rho,1}+\dots +d_{\rho,v_\rho}\geq 0$. Since $d_{\rho,1}\geq d_{\rho,2}\geq \dots \geq d_{\rho,v_\rho}\geq 0$, it suffices to show that $d=1$. For this, the locally-free sheaves $\mathcal{R}_\sigma$ for $\sigma\in \Irr(G)$ provide a $\ZZ$-basis for the Grothendieck group $K(Y)$ of coherent sheaves on $Y$, so the line bundles $\det(\mathcal{R}_\sigma)$ for $\sigma\neq \rho_0$ generate $\Pic(Y)$ over $\ZZ$. However, $\det(\mathcal{R}_\sigma\vert_{\ell})\cong \mathcal{O}_{\mathbb{P}^1}$ for $\sigma\neq \rho$ by the previous paragraph, so $\det(\mathcal{R}_\rho\vert_\ell)$ is the only one of the given generators of $\Pic(Y)$ that has nonzero-degree on the curve $\ell$. If $d>1$, then the intersection pairing on $Y$ would not be perfect, a contradiction.
 \end{proof}

 \begin{proposition}
 \label{prop:contractP1}
 Suppose that the locus $Z_\rho\subseteq Y$ contains a curve $\ell\cong \mathbb{P}^1$ that is contracted by the morphism $\cont_{W_\rho}$. Then $\ell=Z_\rho$ and
  \[
 \Psi(S_\rho) \cong \det(\mathcal{R}_\rho)^{-1}\vert_{Z_\rho}.
 \]
 \end{proposition}
 \begin{proof}
  For $k\in \mathbb{Z}$, write $\Phi^k(-)=H^k(\Phi(-))$ for the $k$th cohomology $A$-module with respect to the standard $t$-structure on $D^b(A)$. Apply \cite[(2.2)]{CI04} to obtain an isomorphism
  \begin{equation}
  \label{eqn:Phiksum}
\Phi^k\big(\det(\mathcal{R}_\rho)^{-1}\vert_{\ell})\big) \cong 
  \bigoplus_{\sigma\in \Irr(G)} H^k\big(\mathcal{R}_\sigma\otimes \det(\mathcal{R}_\rho)^{-1}\vert_{\ell}\big)\otimes \sigma.
  \end{equation}
  Lemma~\ref{lem:deg0or1} shows that $\det(\mathcal{R}_\rho)^{-1}\vert_\ell \cong \det(\mathcal{R}_\rho^\vee)\vert_\ell \cong \mathcal{O}_{\mathbb{P}^1}(-1)$ and that $\mathcal{R}_\sigma\vert_\ell \cong \mathcal{O}_{\mathbb{P}^1}^{\oplus \dim(\sigma)}$ for $\sigma\neq \rho$. It follows that for $\sigma\neq \rho$ and for all $k\in \mathbb{Z}$, we have \[
  H^k\big(\mathcal{R}_\sigma\otimes \det(\mathcal{R}_\rho)^{-1}\vert_{\ell}\big)
   \cong 
   H^k\big(\mathcal{O}_{\mathbb{P}^1}(-1)\big)^{\oplus \dim(\sigma)}= 0.
   \]
 On the other hand, we use Lemma~\ref{lem:deg0or1} to compute
 \[
 H^k\big(\mathcal{R}_\rho\otimes \det(\mathcal{R}_\rho)^{-1}\vert_{\ell}\big)
 \cong 
 H^k(\mathcal{O}_{\mathbb{P}^1}) \oplus H^k\big(\mathcal{O}_{\mathbb{P}^1}(-1)\big)^{\oplus \dim(\rho)-1} 
 \cong 
 \left\{\begin{array}{cr} \CC & \text{ for } k=0 \\ 0 & \text{ otherwise.}
 \end{array}\right.
 \]
 Substituting back into \eqref{eqn:Phiksum} shows that $\Phi^k(\det(\mathcal{R}_\rho)^{-1}\vert_{\ell}))=0$ for $k\neq 0$, hence 
\[
\Phi\big(\det(\mathcal{R}_\rho)^{-1}\vert_{\ell}\big)\cong \Phi^0\big(\det(\mathcal{R}_\rho)^{-1}\vert_{\ell}\big)\cong H^0\big(\mathcal{R}_\rho\otimes \det(\mathcal{R}_\rho)^{-1}\vert_{\ell}\big)\otimes \rho\cong \CC\otimes \rho= S_\rho
 \]
 as a complex concentrated in degree zero. Applying $\Psi$ gives 
 \[
 \Psi(S_\rho)\cong \det(\mathcal{R}_\rho)^{-1}\vert_{\ell}.
 \]
 In particular, the support of the object $\Psi(S_\rho)$ is the curve $\ell \subseteq Z_\rho$. However, Proposition~\ref{prop:support} shows that $\supp \Psi(S_\rho)=Z_\rho$, so $\ell = Z_\rho$ as Zariski closed sets.
    \end{proof}

 \section{Case-by-case analysis of the walls}
 The stability parameter $\eta$ that determines the line bundle $L_C(\eta)$ from \eqref{eqn:LC+eta} is general in the wall $W_\rho=\overline{C}\cap (\theta_\rho=0)$ of the chamber $C$, so $L_C(\eta)$ lies in the relative interior of a codimension-one face of the nef cone of $Y$. It follows that $\cont_{W_\rho}$ cannot be factored as the composition of more than one nontrivial birational morphism. Following \cite[Definition~3.5]{CI04}, we say that $W_\rho$ is a wall of \emph{type $\0$, $\I$, $\II$, or $\III$} as follows:
 \begin{itemize}
     \item  type $\0$ if $\cont_{W_\rho}$ is an isomorphism.
     \item type $\I$ if $\cont_{W_\rho}$ contracts a curve to a point.
     \item type $\II$ if $\cont_{W_\rho}$ contracts a surface to a point.
     \item type $\III$ if $\cont_{W_\rho}$ contracts a surface to a curve.
 \end{itemize}
 This classification allows us to study the objects $\Psi(S_\rho)$ for which $W_\rho$ is a wall of the chamber $C$ according to the type of the wall.

 We begin this case-by-case analysis by studying walls of type $\I$ and $\III$.

\begin{proposition}
\label{prop:typeIorIII}
  Let $W_\rho$ be a wall of type $\I$ or $\III$. 
  The locus $Z_\rho$ in $Y$ is isomorphic to $\mathbb{P}^1$, and
  \[
 \Psi(S_\rho) \cong \det(\mathcal{R}_\rho)^{-1}\vert_{Z_\rho}.
 \]
 In particular, $\Psi(S_\rho)$ is a sheaf and hence $H^k\big(\Psi(S_\rho)\big)=0$ for $k\neq 0$.
 \end{proposition}
 \begin{proof}
  By Proposition~\ref{prop:contractP1}, it suffices to show that the locus $Z_\rho$ contains a curve $\ell\cong \mathbb{P}^1$ that is contacted by $\cont_{W_\rho}$. The morphism $\cont_{W_\rho}$ contracts all curves whose numerical class lies in the extremal ray of the Mori cone of curves $\overline{\text{NE}}(Y)$ that is dual to the codimension-one face of $\Nef(Y)$ containing $L_C(\eta)$. The exceptional locus of $\cont_{W_\rho}$ is contained in the locus of $\eta$-strictly semistable points, and hence so is each curve in the numerical class of the extremal ray. Lemma~\ref{lem:curveintauinverse0} shows that a curve $\ell'$ in this numerical class also lies in $\tau^{-1}(\pi(0))$, so $\ell'\subseteq Z_\rho$. Since $W_\rho$ is of type $\I$ or $\III$, the morphism $\cont_{W_\rho}$ is of fibre dimension at most one,  and it is a crepant contraction from the nonsingular variety $Y$,
  so the underlying reduced scheme of each irreducible component of the fibre $\cont_{W_\rho}^{-1}(y)$ is  $\mathbb{P}^1$ by \cite[Corollary~1.5.1]{AW98}. In particular, the curve $\ell'$ is an irreducible component of the fibre that contains $\ell:=\mathbb{P}^1$ as a closed subscheme. The result now follows from Proposition~\ref{prop:contractP1} because $\ell\subseteq \ell'\subseteq Z_\rho$. 
 \end{proof}

 Next, we consider the type $\II$ case.
 
 \begin{proposition}
 \label{prop:notypeII}
 The chamber $C$ does not have a wall of type $\II$ of the form $W_\rho=\overline{C}\cap(\theta_\rho=0)$.
 \end{proposition}
 \begin{proof}
 Suppose otherwise, so a surface $S$ is contracted to a point by $\cont_{W_\rho}$. The crepant resolution $\tau\colon Y\to \mathbb{C}^3/G$ factors via $\cont_{W_\rho}$, so $\cont_{W_\rho}$ is a crepant resolution of a canonical threefold.  Reid~\cite{Reid79} showed that $S$ is Gorenstein with $\omega_S^{-1}$ ample. 
 We claim that $S$ contains a nonsingular rational curve $\ell\cong \mathbb{P}^1$. When $S$ is normal, this follows from Hidaka--Watanabe~\cite[Theorems~2.2 and~3.4]{HW81}, and for non-normal $S$, it follows from Reid~\cite[Proposition 4.9]{Reid94} (compare also Mori~\cite[(3.35)]{Mori82}). Then $\ell\subseteq Z_\rho$ because $S\subseteq Z_\rho$, so Proposition~\ref{prop:contractP1} gives $\ell=Z_\rho$. This is absurd because $S\subseteq Z_\rho$.
 \end{proof}

Finally, suppose $C$ has a wall of type $\0$ of the form $W_\rho=\overline{C}\cap (\theta_\rho=0)$, and choose $\eta\in \relint(W_\rho)$ as before. For each $z\in Z_\rho$, recall from \eqref{eqn:sesV_z} that there is a destabilising short exact sequence
 \[
 0\longrightarrow S_\rho\longrightarrow V_z\longrightarrow Q_z\longrightarrow 0
 \]
 of $\eta$-semistable $A$-modules.
  
  \begin{lemma}
 \label{lem:Hom2}
  For any wall $W_\rho$ of type $\0$, we have  
 \[
\Hom_{D(Y)}^2\big(\Psi(S_\rho),\mathcal{O}_y\big) = 0 \text{ for all }y\in Y.
 \]
 \end{lemma}
 \begin{proof}
 This follows for $y\not\in \supp(\Psi(S_\rho))$ by \cite[Lemma~5.3]{BM02}. Proposition~\ref{prop:support} shows that it's enough to prove the result for $y\in Z_\rho$. 
 Each irreducible component of $Z_\rho$ has dimension at least one 1 by Proposition~\ref{prop:support} and Lemma~\ref{lem:Hom3}, so there exists $z\in Z_\rho$ such that $z\neq y$. Apply $\Hom_A(V_y,-)$ to the above destabilising short exact sequence to obtain an exact sequence
 \[
 \Hom_A(V_y,Q_z)\longrightarrow \Ext^1_A(V_y,S_\rho)\longrightarrow \Ext^1_A(V_y,V_z)
 \]
of $\CC$-vector spaces. Since $y\neq z$, we have $\Ext^1_A(V_y,V_z)=0$ precisely as in \eqref{eqn:ExtBKR}. Now, suppose for a contradiction that $\Hom_A(V_y,Q_z)\neq 0$. Corollary~\ref{cor:Lin} implies that $V_y$ and $V_z$ are $\Seshadri$-equivalent, so the points $y$ and $z$ of $Y$ are identified by the VGIT morphism $Y \to \mathcal{M}_\eta(A,v)$. However, $W_\rho$ is a wall of type $\0$, so this morphism is injective by Craw--Ishii~\cite[Lemma~3.3]{CI04}, giving $y=z$,
 a contradiction. Thus $\Hom_A(V_y,Q_z)=0$ after all, and the above exact sequence gives $\Ext_A^1(V_y,S_\rho)=0$, so  
\[
\Hom_{D(Y)}^2\big(\Psi(S_\rho),\mathcal{O}_y\big) \cong \Ext^2_A(S_\rho, V_y) \cong\Ext_A^1(V_y,S_\rho)^* =0 
\]
for all $y\in Z_\rho$ as required. 
 \end{proof}

 \begin{proposition}
\label{prop:type0}
 For any wall $W_\rho$ of type $\0$, we have $H^k\big(\Psi(S_\rho)\big)=0$ for all $k\neq 0$.
 Moreover, every irreducible component of $\supp\Psi(S_\rho) = Z_\rho$ has dimension two.
\end{proposition}  
\begin{proof} 
The vanishing of $\Hom_{D(Y)}^2(\Psi(S_i),\mathcal{O}_y)$ for all $y\in Y$ from Lemma~\ref{lem:Hom2} allows us to strengthen the argument in the proof of Lemma~\ref{lem:Hom3} to deduce that $\Psi(S_\rho)$ is quasi-isomorphic to a complex of locally-free sheaves of the form
 \[
 P^{-1}\stackrel{f}{\longrightarrow} P^{0}.
 \]
 The object $\Psi(S_\rho)$ has proper support by Proposition~\ref{prop:support}, so 
 $\ker(f)=0$. We may therefore identify $\Psi(S_\rho)$ with the cokernel of a map of locally free sheaves of the same rank, so it is a coherent sheaf, giving $H^k\big(\Psi(S_\rho)\big)=0$ for all $k\neq 0$. The assertion about the dimension of the support of $\Psi(S_\rho)$ follows from \cite[Corollary~5.5]{BM02} and Proposition~\ref{prop:support}.  
\end{proof}

\begin{proof}[Proof of Theorem~\ref{thm:mainintro}]
Let $\rho\neq \rho_0$. If the hyperplane $(\theta_\rho=0)$ intersects the chamber $C$, then 
 \[
 H^k\big(\Psi(S_\rho)\big)=0 \quad\text{for }k\neq -1
 \]
 and every irreducible component of the support of $\Psi(S_\rho)$ is a surface by Proposition~\ref{prop:type1}. It follows that the complex $\Psi(S_\rho)$ is concentrated in degree $-1$, so shifting it to the right by one produces the new object $\Psi(S_\rho)[-1]$ which is a coherent sheaf. Also, since $\Psi$ is an equivalence of derived categories, we have
 \[
 \End_{D(Y)}\big(\Psi(S_\rho)\big) \cong \End_A(S_\rho) \cong \CC,
 \]
 so the support of $\Psi(S_\rho)$ is connected. This completes the analysis in case \ensuremath{(-)}.

 Otherwise, Lemma~\ref{lem:C} shows that $\overline{C}\cap (\theta_\rho=0)$ is a wall of the form $W_\rho$. This wall cannot be of type $\II$ by Proposition~\ref{prop:notypeII}, so it must be of type $\0$, $\I$ or $\III$. Suppose first that $W_\rho$ is of type $\I$ or $\III$. Then Proposition~\ref{prop:typeIorIII} shows that the locus $Z_\rho$ in $Y$ is isomorphic to $\mathbb{P}^1$, and
  \[
 \Psi(S_\rho) \cong \det(\mathcal{R}_\rho)^{-1}\vert_{Z_\rho}.
 \]
 In particular, $\Psi(S_\rho)$ is a sheaf and hence 
 \[
 H^k\big(\Psi(S_\rho)\big)=0\quad \text{for }k\neq 0.
 \]
 Lemma~\ref{lem:deg0or1} shows that $\det(\mathcal{R}_\rho)$ has degree one on the curve $\ell:= Z_\rho$, so $\det(\mathcal{R}_\rho)^{-1}\vert_{Z_\rho}\cong \mathcal{O}_{\ell}(-1)$. This completes the analysis in case \ensuremath{(0)}. 
 
 Finally, suppose that $W_\rho$ is of type $\0$. 
  Proposition~\ref{prop:type0} shows that 
  \[
 H^k\big(\Psi(S_\rho)\big)=0 \quad\text{for }k\neq 0,
 \]
 so Conjecture~\ref{conj:CLintro} holds, with $k=-1$ if and only if the hyperplane $(\theta_\rho=0)$ intersects $C$. 
 Moreover, Proposition~\ref{prop:type0} also shows that every irreducible component of $\supp\Psi(S_\rho)$ has dimension two. The same argument as in case \ensuremath{(-)} above shows that the support of $\Psi(S_\rho)$ is connected.
\end{proof}

 \begin{proof}[Proof of Theorem~\ref{thm:BCQconjectureintro}]
  Case \ensuremath{(0)} from Theorem~\ref{thm:mainintro} occurs only when the nontrivial representation $\rho$ determines a wall $W_\rho$ of type $\I$ or $\III$, in which case the result follows from Proposition~\ref{prop:typeIorIII}.
  \end{proof}

\section{A sign coherent matrix}
\label{sec:signcoherent}
 As an application of Theorem~\ref{thm:mainintro}, we prove Theorem~\ref{thm:signcoherentintro}.  
 Write $\Pic(Y)_{\mathbb{Q}}$
  for the rational Picard group of $Y$. The \emph{linearisation map} for $Y=\ghilb(\mathbb{C}^3)$ is the linear map of rational vector spaces
  \begin{equation}
      \label{eqn:linearisation}
   L_C\colon \Theta\longrightarrow \Pic(Y)_{\mathbb{Q}}
  \end{equation}
   that sends $\eta\in \Theta$ to $\bigotimes_{\sigma\neq \rho_0} \det(\mathcal{R}_\sigma)^{\otimes \eta_\sigma}$ as in \eqref{eqn:LC+eta}. 

   Let $K(Y)$ and $K_0(Y)$ denote the Grothendieck groups of the category of coherent sheaves on $Y$ and the category of coherent sheaves on $Y$ with support in $\tau^{-1}(\pi(0))$ respectively. Consider the topological filtration 
 \[
K(Y)_{\mathbb{Q}}=F^0\supseteq F^1\supseteq F^2\supseteq \{0\}
  \]
 where $F^m$ denotes the subspace spanned by sheaves with support of codimension at least $m$, and the dimension filtration
  \[
 \{0\}\subseteq F_0\subseteq F_1 \subseteq F_2 = K_0(Y)_{\mathbb{Q},
 }\]
  where $F_m$ is spanned by sheaves with support of dimension at most $m$. Craw--Ishii~\cite[Proposition~5.1]{CI04} implies that the Euler form gives a perfect pairing $\chi\colon K(Y)_{\QQ}\times K_{0}(Y)_{\QQ}\to \QQ$, where 
 \[
 K_0(Y)_{\QQ} = \bigoplus_{\rho\in \Irr(G)} \QQ [\Psi(S_\rho)],
 \]
 and $F^1=F_0^\perp$ and $F^2=F_1^\perp$ with respect to $\chi$. In addition, \cite[Corollary~5.2]{CI04} shows that the derived equivalence induces an isomorphism on Grothendieck groups that identifies $\Theta$ with $F^1$. Since $\Pic(Y)_{\mathbb{Q}}\cong F^1/F^2$,  we obtain a short exact sequence
       \begin{equation}
       \label{eqn:sesPic}
 \begin{tikzcd}
 0\ar[r] & F^2\ar[r] & F^1\ar[r,"L_C"] & \Pic(Y)_{\QQ}\ar[r] & 0
  \end{tikzcd}
  \end{equation}
of $\QQ$-vector spaces.  The \emph{Gale dual} map to $L_C$ is the map $G_C$ in the dual short exact sequence 
      \begin{equation}
 \begin{tikzcd}
   \label{eqn:sesPicdual}
 0 & F_2/F_1\ar[l] & F_2/F_0\ar[l,swap,"G_C"] & \Pic(Y)_{\QQ}^*\ar[l] & 0\ar[l]
  \end{tikzcd}
  \end{equation}
of $\QQ$-vector spaces. Recall the geometric interpretation of $G_C$ by the first author \cite[Theorem~1.1]{C21}:

   \begin{proposition}
   \label{prop:GaleDual}
  The classes $\{[\Psi(S_\rho)] \mid \rho\neq \rho_0\}$ provide a basis for the $\QQ$-vector space $F_2/F_0$. In addition, the $\QQ$-linear map $G_C\colon F_2/F_0 \to F_2/F_1$ is determined by  \begin{equation}
     \label{eqn:GCintro}
G_{\chamber}\big([\Psi(S_\rho)]\big) = \sum_{k\in \mathbb{Z}} (-1)^{k}\sum_{V} \ell_V\big(H^k\big(\Psi(S_\rho)\big)\big)[\mathcal{O}_{V}] 
 \end{equation}
  for $\rho\neq \rho_0$, where the second sum runs over all irreducible surfaces $V$ in the support of $H^k(\Psi(S_\rho))$, and where $\ell_V(-)$ computes the length of the stalk of the sheaf at the generic point of $V$.   \end{proposition}
  
  \begin{proof}
 Apply \cite[Theorem~1.1\one]{C21} and extend scalars over $\QQ$ to see that $F_2/F_0$ is isomorphic to the quotient of the vector space $\bigoplus_{\rho\in \Irr(G)} \QQ[\Psi(S_\rho)]$ by the 1-dimensional vector subspace spanned by the element $\sum_{\rho\in \Irr(G)} (\dim\rho) [\Psi(S_\rho)]$. Since $\dim \rho_0=1$, we may use the relation 
 \[
 \sum_\rho (\dim \rho) [\Psi(S_\rho)]=0
 \]
 to solve for the class $[\Psi(S_{\rho_0})]$.  This proves the first statement, while the second statement follows immediately from \cite[Theorem~1.1\two]{C21} after extending scalars over $\QQ$.
  \end{proof}

  List the irreducible components of the exceptional locus of  $\tau\colon Y\to \mathbb{C}^3/G$ as $E_1, \dots, E_m$, where $E_1, \dots, E_l$ are proper for some $\ell\leq m$, and $E_{l+1}, \dots , E_m$ are not proper. The following result is well known.

  \begin{lemma}
 The surface $E_i$ in $Y$ is proper if and only if $E_i\subseteq \tau^{-1}(\pi(0))$.
  \end{lemma}
  \begin{proof}
    The fibre $\tau^{-1}(\pi(0))$ is proper, and so is each closed subvariety $E_i$ in $\tau^{-1}(\pi(0))$.  Conversely, since $\mathbb{C}^3/G$ is affine, the image of the proper irreducible subvariety $E_i$ under the proper morphism $\tau$ is a closed point  $t \in \CC^3/G$.
    Suppose that $t \ne \pi(0)$. 
    The $\mathbb{C}^*$-equivariance of $\tau : Y \to \mathbb{C}^3/G$ 
    implies that
    the $\mathbb{C}^*$-orbit
    $\mathbb{C}^* \cdot E_i \subset Y$ has dimension three, and that
    the surfaces $\lambda \cdot E_i$
    with $\lambda \in \mathbb{C}^*$ are all contracted by $\tau$.
    This contradicts the fact that $\tau$ is birational, so $E_i\subseteq \tau^{-1}(\pi(0))$ after all.
  \end{proof}

 \begin{lemma}
 \label{lem:F2/F1basis}
  The classes $[\mathcal{O}_{E_1}], \dots , [\mathcal{O}_{E_l}]$ provide a basis for the rational vector space $F_2/F_1$.
  \end{lemma}
  \begin{proof}
  For any sheaf $F$ on $Y$ whose support lies in $\tau^{-1}(\pi(0))$, its class in $K_0(Y)$ satisfies 
  \[
  [F] = \sum_{1\leq i\leq l} \ell_{E_i}(F)[\mathcal{O}_{E_i}] \mod F_1
  \]
  by \cite[Expose X, Proposition~1.1.2]{SGA6}, so the classes $[\mathcal{O}_{E_1}], \dots , [\mathcal{O}_{E_l}]$ span $F_2/F_1$. 
  To show linear independence, recall from \cite[Example~18.3.11]{Fulton98} that for any integral subscheme $Z \subset \tau^{-1}(\pi(0))$ of maximal dimension $2$, the length
  $\ell_Z(\mathcal{F})$ of a coherent sheaf $\mathcal{F}$ supported on $\tau^{-1}(\pi(0))$
  along the generic point of $Z$
  extends to a group homomorphism
  \[
  \ell_Z \colon K_0(Y)_\QQ \to \QQ.
  \]
  Suppose that $\alpha = \sum_i a_i [\mathcal{O}_{E_i}] \in F_1$.
  Then $a_i = \ell_{E_i}(\alpha) = 0$
  for all $1\leq i\leq l$ as required.
  \end{proof}

 \begin{proof}[Proof of Theorem~\ref{thm:signcoherentintro}]
  Theorem~\ref{thm:mainintro} provides a unique value $k\in \{-1,0\}$ that depends on $\rho$ for which $H^k(\Psi(S_\rho))\neq 0$, so the first sum from \eqref{eqn:GCintro} collapses. Expanding in terms of the basis of $F_2/F_1$ from  Lemma~\ref{lem:F2/F1basis} gives
   \begin{equation}
     \label{eqn:GCbasis}
G_{\chamber}\big([\Psi(S_\rho)]\big) = \left\{\begin{array}{rr}
\sum_{1\leq i\leq l} \ell_{E_i}\big(H^0\big(\Psi(S_\rho)\big)\big)[\mathcal{O}_{E_i}]  & \text{if }k=0; \\
-\sum_{1\leq i\leq l} \ell_{E_i}\big(H^{-1}\big(\Psi(S_\rho)\big)\big)[\mathcal{O}_{E_i}]  & \text{if }k=-1,
\end{array}\right.
 \end{equation}
  for each $\rho\neq \rho_0$. As a result, the $[\Psi(S_\rho)]$-column of the matrix representating $G_C$ in the chosen bases therefore satisfies either:
 \begin{enumerate}
 \item[\ensuremath{(+)}] every nonzero entry is positive; this happens if and only if $k=0$ and the support of $\Psi(S_\rho)$ has dimension two;
 \item[\ensuremath{(0)}] the column is the zero vector; this happens if and only if the support of $\Psi(S_\rho)$ has dimension one; or
  \item[\ensuremath{(-)}] every nonzero entry is negative; this happens if and only if $k=-1$ and the support of $\Psi(S_\rho)$ has dimension two.
 \end{enumerate}
 Therefore the matrix representing $G_C$ is sign coherent. \end{proof}

 \begin{example}[\textbf{Dihedral case}]
 \label{exa:BentoBox}
 Let $G\subset \SL(3,\CC)$ be the dihedral group of order 24 generated by 
 \[
 \begin{pmatrix} e^{\pi i/6} & 0 & 0 \\ 
 0 & e^{7\pi i/6} & 0 \\ 0 & 0 & e^{2\pi i/3} \end{pmatrix}\quad\text{and}\quad
  \begin{pmatrix} 
 0 & 1 & 0 \\ 
 -1 & 0 & 0 \\ 
 0 & 0 & 1
 \end{pmatrix}.
 \]
  The McKay quiver $Q$ was computed by  
 Nolla de Celis~\cite[Figure~5]{NolladeCelis21}, and the first author made many predictions in \cite[Example~6.5]{C21}, all of which are valid in light of Theorems~\ref{thm:mainintro}, \ref{thm:signcoherentintro} and \ref{thm:BCQconjectureintro}.  In particular, the sign coherent matrix from Theorem~\ref{thm:signcoherentintro} is shown in (6.5) of \emph{loc.~cit.}, and the support of each object $\Psi(S_\rho)$ is computed explicitly in cases $(+)$ and $(-)$ described in \emph{loc.~cit.}. In addition, in case $(0)$ of the same example,  Theorem~\ref{thm:BCQconjectureintro} gives that 
 \[
 \Psi(S_\rho)\cong \det(\mathcal{R}_\rho)^{-1}\vert_{Z_\rho}
 \]
 for the representations $\rho\in \{\rho_1, \rho_{4^-}, \rho_{10^+}, \rho_{10^-}\}$, where each $Z_\rho$ is a nonsingular rational curve. 
  \end{example}

 \begin{example}[\textbf{Trihedral case}]
 \label{exa:trihedralNdC}
 Consider the trihedral subgroup $G\subset \SL(3,\CC)$ obtained from the cyclic group action of type $\frac{1}{13}(1,3,9)$; more explicitly, $G$ is the group
 of order 39 generated by
 \[
 \begin{pmatrix} e^{2 \pi i/13} & 0 & 0 \\ 
 0 & e^{6 \pi i/13} & 0 \\ 0 & 0 & e^{18 \pi i/13} \end{pmatrix}\quad\text{and}\quad
  \begin{pmatrix} 
 0 & 1 & 0 \\ 
 0 & 0 & 1 \\ 
 1 & 0 & 0
 \end{pmatrix}.
 \]
 Following Nolla de Celis~\cite[Section~4]{NolladeCelis21}, we list the 7 irreducible representations as
 \begin{equation}
    \label{eqn:irreps}
 \{\rho_0, \rho_{0}^\prime, \rho_{0}^{\prime\prime}, V_{1}, V_2, V_4, V_7\},
  \end{equation}
 where $\rho_{0}^\prime, \rho_{0}^{\prime\prime}$ are nontrivial of dimension one, and each $V_i$ has dimension three.
 
 The exceptional fibre $\tau^{-1}(\pi(0))$ in $Y=\ghilb(\CC^3)$ consists of two proper divisors $E_1$ and $E_2$. List the indecomposable summands of the tautological bundle as $\mathcal{R}_{0}, \mathcal{R}_{0}^\prime, \mathcal{R}_{0}^{\prime\prime}, \mathcal{R}_{1}, \mathcal{R}_2, \mathcal{R}_4, \mathcal{R}_7$ with $\mathcal{R}_{0}\cong \mathcal{O}_Y$. The key calculation of Nolla de Celis~\cite[Proposition~2]{NolladeCelis21} associates one relation between the determinants of the tautological bundles for each surface $E_1$ and $E_2$, namely
   \begin{align}
 \mathcal{R}_0^\prime\otimes \mathcal{R}_0^{\prime\prime} & \cong\det(\mathcal{R}_4)   & \text{for }E_1, \label{eqn:rels1} \\
 \det(\mathcal{R}_1) & \cong  \det(\mathcal{R}_2) & \text{for }E_2. \label{eqn:rels2} 
 \end{align}
 In addition, the line bundles $\det(\mathcal{R}_{2}), \det(\mathcal{R}_4), \det(\mathcal{R}_7)$ are independent in $\Pic(Y)$, and we extend this to a $\ZZ$-basis by breaking symmetry in equation \eqref{eqn:rels1} by adding $\mathcal{R}_0^\prime$ as the first basis element.

 To compute the matrix defining the linearisation map $L_C$ as in \cite[Definition~6.1]{C21}, list the basis elements of $\Theta$ for $\rho\neq \rho_{0}$ in the same order as in \eqref{eqn:irreps} and then use \eqref{eqn:rels1}-\eqref{eqn:rels2} to compute that 
 \[
\setlength{\arraycolsep}{2pt}
   \renewcommand{\arraystretch}{0.8}
L_C=\mbox{\footnotesize$\left(\begin{array}{crcccc}
1 & -1 & 0 & 0 & 0 & 0  \\
0 & 0 & 1 & 1 & 0 & 0  \\
0 & 1 & 0 & 0 & 1 & 0  \\
0 & 0 & 0 & 0 & 0 & 1    
\end{array}\right).$}
\]
For example, the second column expresses $\mathcal{R}_0^{\prime\prime}$ as $\det(\mathcal{R}_4)\otimes (\mathcal{R}_0^{\prime})^{-1}$. The exponents on the relations \eqref{eqn:rels1} and \eqref{eqn:rels2} determine the first and second columns of the matrix defining the inclusion of $\ker(L_C)$ in $\Theta$, and the transpose of that matrix, namely 
 \begin{equation}
    \label{eqn:Ktbento}
\setlength{\arraycolsep}{2pt}
   \renewcommand{\arraystretch}{0.8}
G_C=\mbox{\footnotesize$\left(\begin{array}{rrcrrc}
1 & 1 &  0 &  0 & -1  & 0 \\
0 & 0 &  1 & -1 & 0  & 0  
\end{array}\right)$}
\end{equation}
 defines the Gale dual map expressed in the basis $\{[\mathcal{O}_{E_1}], [\mathcal{O}_{E_2}]\}$ of $F_2/F_1$. Note that $G_C$ is sign-coherent as in Theorem~\ref{thm:signcoherentintro}. The signs in the columns of $G_C$ characterise the trichotomy:

 \begin{enumerate}
 \item[\ensuremath{(+)}] The columns of $G_C$ containing a positive entry are indexed by representations $\rho_{0}^\prime, \rho_{0}^{\prime\prime}, V_1$ of type \ensuremath{(+)}. The positive entries in the columns for $\rho_{0}^\prime, \rho_{0}^{\prime\prime}$ appear in the first row, so Proposition~\ref{prop:GaleDual} shows that 
 $\Psi(S_{0^\prime}), \Psi(S_{0^{\prime\prime}})$ are sheaves with support $Z_{\rho_0^\prime}=Z_{\rho_0^{\prime\prime}}=E_1$, while the positive entry for the column $V_1$ is in row two, so the sheaf $\Psi(S_{V_1})$ has support $Z_{V_1}=E_2$. 
 \item[\ensuremath{(0)}] The column of $G_C$ that is the zero vector indicates that $\Psi(S_{V_7})\in \ker(G_C)=F_1/F_0$, that is, $\Psi(S_{V_7})$ is supported on curves. In fact, \cite[Figure~11]{NolladeCelis21} shows that $\mathcal{R}_7$ is the unique tautological bundle that has nonzero degree on the curve denoted $C_2$; this agrees with Lemma~\ref{lem:deg0or1}. In particular, Proposition~\ref{prop:contractP1} implies that
 \[
 \Psi(S_{V_7})\cong \det(\mathcal{R}_7)^{-1}\vert_{C_2}\cong\mathcal{O}_{C_2}(-1)
 \]
 as in  Theorem~\ref{thm:mainintro}\ensuremath{(0)}.
 \item[\ensuremath{(-)}]  The columns of $G_C$ containing a negative entry are indexed by representations $V_2, V_4$ of type \ensuremath{(-)}. Again, by examining the rows containing the negative entries, Proposition~\ref{prop:GaleDual} implies that the sheaves $\Psi(S_{V_2})[-1]$ and $\Psi(S_{V_4})[-1]$ have support equal to $E_2$ and $E_1$ respectively.
 \end{enumerate}
 
 In addition, the `Socle' column of \cite[Table~10]{NolladeCelis21} shows that  neither $V_2$ nor $V_4$ lies in the socle of any $G$-cluster, so the coordinate hyperplanes $(\theta_\rho=0)$ for  $\rho\in \{V_2, V_4\}$ intersect the GIT chamber $C$, which is consistent with Theorem~\ref{thm:mainintro}\ensuremath{(-)}. Conversely, the representations $\rho_{0}^\prime, \rho_{0}^{\prime\prime}$ lie in the socle of $G$-clusters parametrised by $E_1$, the representation $V_{1}$ lies in the socle of $G$-clusters parametrised by $E_2$, and 
 $V_{7}$ lies in the socle of the $G$-clusters parametrised by the curve $C_2$ from \cite[Figure~9]{NolladeCelis21}, so 
 the coordinate hyperplanes  $(\theta_\rho=0)$ for each $\rho\in \{\rho_{0}^\prime, \rho_{0}^{\prime\prime}, V_1, V_7\}$ provide supporting hyperplanes for the chamber $C$. These calculations are consistent with cases $(+)$ and $(0)$ of Theorem~\ref{thm:mainintro}.
 
\end{example}

 \small{
\bibliographystyle{plain}


\begin{thebibliography}{10}

\bibitem{SGA6}
{\em Th\'eorie des intersections et th\'eor\`eme de {R}iemann-{R}och}.
\newblock Lecture Notes in Mathematics, Vol. 225. Springer-Verlag, Berlin-New
  York, 1971.
\newblock S{\'e}minaire de G{\'e}om{\'e}trie Alg{\'e}brique du Bois-Marie
  1966--1967 (SGA 6), Dirig{\'e} par P. Berthelot, A. Grothendieck et L.
  Illusie. Avec la collaboration de D. Ferrand, J. P. Jouanolou, O. Jussila, S.
  Kleiman, M. Raynaud et J. P. Serre.

\bibitem{AW98}
M.~Andreatta and J.~Wi{\'s}niewski.
\newblock On contractions of smooth varieties.
\newblock {\em J. Algebr. Geom.}, 7(2):253--312, 1998.

\bibitem{BCZ17}
A.~Bayer, A.~Craw, and Z.~Zhang.
\newblock Nef divisors for moduli spaces of complexes with compact support.
\newblock {\em Selecta Math. (N.S.)}, 23(2):1507--1561, 2017.

\bibitem{BCQ15}
R.~Bocklandt, A.~Craw, and A.~Quintero~V\'{e}lez.
\newblock Geometric {R}eid's recipe for dimer models.
\newblock {\em Math. Ann.}, 361(3-4):689--723, 2015.

\bibitem{BCQ21}
R.~Bocklandt, A.~Craw, and A.~Quintero~V\'{e}lez.
\newblock Correction to: {G}eometric {R}eid's recipe for dimer models.
\newblock {\em Math. Ann.}, 380(1-2):911--913, 2021.

\bibitem{BKR01}
T.~Bridgeland, A.~King, and M.~Reid.
\newblock The {M}c{K}ay correspondence as an equivalence of derived categories.
\newblock {\em J. Amer. Math. Soc.}, 14(3):535--554, 2001.

\bibitem{BM02}
T.~Bridgeland and A.~Maciocia.
\newblock Fourier-{M}ukai transforms for {$K3$} and elliptic fibrations.
\newblock {\em J. Algebraic Geom.}, 11(4):629--657, 2002.

\bibitem{CCL17}
S.~Cautis, A.~Craw, and T.~Logvinenko.
\newblock Derived {R}eid's recipe for abelian subgroups of {${\rm
  SL}_3(\mathbb{C})$}.
\newblock {\em J. Reine Angew. Math.}, 727:1--48, 2017.

\bibitem{CL09}
S.~Cautis and T.~Logvinenko.
\newblock A derived approach to geometric {M}c{K}ay correspondence in dimension
  three.
\newblock {\em J. Reine Angew. Math.}, 636:193--236, 2009.

\bibitem{CL14}
S.~Cautis and T.~Logvinenko.
\newblock Erratum to ``{A} derived approach to geometric {M}c{K}ay
  correspondence in dimension three'' ({J}. reine angew. {M}ath. 636 (2009),
  193--236).
\newblock {\em J. Reine Angew. Math.}, 689:243--244, 2014.

\bibitem{Crawthesis}
A.~Craw.
\newblock The {M}c{K}ay correspondence and representations of the {M}c{K}ay
  quiver, 2001.
\newblock PhD thesis, University of Warwick, Available from
  \url{https://sites.google.com/view/alastair-craw/}.

\bibitem{C21}
A.~Craw.
\newblock Gale duality and the linearisation map for noncommutative crepant
  resolutions, 2021.
\newblock In submission, available from \url{https://arxiv.org/pdf/2109.09565}.

\bibitem{CI04}
A.~Craw and A.~Ishii.
\newblock Flops of {$G$}-{H}ilb and equivalences of derived categories by
  variation of {GIT} quotient.
\newblock {\em Duke Math. J.}, 124(2):259--307, 2004.

\bibitem{CIK18}
A.~Craw, Y.~Ito, and J.~Karmazyn.
\newblock Multigraded linear series and recollement.
\newblock {\em Math. Z.}, 289(1-2):535--565, 2018.

\bibitem{CQV15}
A.~Craw and A.~Quintero~V\'elez.
\newblock Cohomology of wheels on toric varieties.
\newblock {\em Hokkaido Math. J.}, 44(1):47--79, 2015.

\bibitem{Fulton98}
W.~Fulton.
\newblock {\em Intersection Theory}, volume~2 of {\em Ergebnisse der Mathematik
  und ihrer Grenzgebiete, 3. Folge}.
\newblock Springer-Verlag, 2 edition, 1998.

\bibitem{GroHilb}
A.~Grothendieck.
\newblock Techniques de construction et th{\'e}oremes d'existence en
  g{\'e}om{\'e}trie alg{\'e}brique. {IV}: {Les} schemas de {Hilbert}.
\newblock \emph{Sem. {Bourbaki}} 13(1960/61), {No}. 221, 28 p., 1961.

\bibitem{GM02}
J.~Guo and R.~Mart{\'i}nez-Villa.
\newblock Algebra pairs associated to {McKay} quivers.
\newblock {\em Comm. Algebra}, 30(2):1017--1032, 2002.

\bibitem{HW81}
F.~Hidaka and K.~Watanabe.
\newblock Normal {Gorenstein} surfaces with ample anti-canonical divisor.
\newblock {\em Tokyo J. Math.}, 4:319--330, 1981.

\bibitem{IU16}
A.~Ishii and K.~Ueda.
\newblock Dimer models and crepant resolutions.
\newblock {\em Hokkaido Math. J.}, 45(1):1--42, 2016.

\bibitem{ItoNakajima00}
Y.~Ito and H.~Nakajima.
\newblock Mc{K}ay correspondence and {H}ilbert schemes in dimension three.
\newblock {\em Topology}, 39(6):1155--1191, 2000.

\bibitem{ItoNakamura99}
Y.~Ito and I.~Nakamura.
\newblock Hilbert schemes and simple singularities.
\newblock In {\em New trends in algebraic geometry ({W}arwick, 1996)}, volume
  264 of {\em London Math. Soc. Lecture Note Ser.}, pages 151--233. Cambridge
  Univ. Press, Cambridge, 1999.

\bibitem{KV00}
M.~Kapranov and E.~Vasserot.
\newblock Kleinian singularities, derived categories and {H}all algebras.
\newblock {\em Math. Ann.}, 316(3):565--576, 2000.

\bibitem{King94}
A.~King.
\newblock Moduli of representations of finite-dimensional algebras.
\newblock {\em Quart. J. Math. Oxford Ser. (2)}, 45(180):515--530, 1994.

\bibitem{Logvinenko10}
T.~Logvinenko.
\newblock Reid's recipe and derived categories.
\newblock {\em J. Algebra}, 324(8):2064--2087, 2010.

\bibitem{McKay80}
J.~McKay.
\newblock Graphs, singularities, and finite groups.
\newblock In {\em The {S}anta {C}ruz {C}onference on {F}inite {G}roups ({U}niv.
  {C}alifornia, {S}anta {C}ruz, {C}alif., 1979)}, volume~37 of {\em Proc.
  Sympos. Pure Math.}, pages 183--186. Amer. Math. Soc., Providence, R.I.,
  1980.

\bibitem{Mori82}
S.~Mori.
\newblock Threefolds whose canonical bundles are not numerically effective.
\newblock {\em Ann. of Math. (2)}, 116(1):133--176, 1982.

\bibitem{NolladeCelis21}
\'A. Nolla~de Celis.
\newblock On {R}eid's recipe for non-abelian groups.
\newblock In {\em Mc{K}ay correspondence, mutation and related topics},
  volume~88 of {\em Adv. Stud. Pure Math.}, pages 95--118. Math. Soc. Japan,
  Tokyo, [2023] \copyright 2023.

\bibitem{Reid79}
M.~Reid.
\newblock Canonical {$3$}-folds.
\newblock In {\em Journ\'ees de {G}\'eometrie {A}lg\'ebrique d'{A}ngers,
  {J}uillet 1979/{A}lgebraic {G}eometry, {A}ngers, 1979}, pages 273--310.
  Sijthoff \& Noordhoff, Alphen aan den Rijn---Germantown, Md., 1980.

\bibitem{Reid94}
M.~Reid.
\newblock Nonnormal del {P}ezzo surfaces.
\newblock {\em Publ. Res. Inst. Math. Sci.}, 30(5):695--727, 1994.

\bibitem{Stacks}
The {S}tacks~{P}roject {A}uthors.
\newblock Stacks {P}roject.
\newblock \url{https://stacks.math.columbia.edu/tag/0AVT}.

\bibitem{VdB04Duke}
M.~Van~den Bergh.
\newblock Three-dimensional flops and non-commutative rings.
\newblock {\em Duke Mathematical Journal}, 122(3):423--455, 2004.

\end{thebibliography}
}
\end{document}